\documentclass[11pt]{article}
\usepackage[utf8]{inputenc}
\usepackage[dvipsnames]{xcolor}
\usepackage{graphicx,amsfonts,amssymb,amsmath,latexsym,amsthm,fullpage, framed,url,tikz,comment,cancel,mathtools, hyperref} 
\usepackage[normalem]{ulem}
\usepackage{centernot}
\usepackage{multicol}
\usepackage{hyperref}
\usepackage{dsfont}
\usepackage[T1]{fontenc}

\usepackage[
backend=biber,
style=numeric,
sorting=nyt,
maxcitenames=2,
maxbibnames=10,
giveninits=true,
]{biblatex}

\bibliography{References} 

\newcommand{\dist}{\mathrm{d}}
\newcommand{\e}{\varepsilon}
\newcommand{\eps}{\varepsilon}
\newcommand{\ceq}{\coloneqq}
\newcommand{\comp}{\mathsf{c}}

\newcommand{\tx}{\tilde{x}}

\newcommand{\cP}{\mathcal{P}}
\newcommand{\cB}{\mathcal{B}}

\newcommand{\R}{\mathbb{R}}

\newcommand{\N}{\mathbb{N}}
\newcommand{\E}{\mathbb{E}}
\newcommand{\prob}{\mathbb{P}}

\newcommand{\ds}{\displaystyle}

\newcommand{\grad}{\nabla}

\DeclareMathOperator*{\argmin}{\arg\min}
\DeclareMathOperator{\Per}{Per}
\DeclareMathOperator{\ProbPer}{ProbPer}

\DeclareUnicodeCharacter{0301}{\'{e}}
\newcommand{\email}[1]{\href{mailto:#1}{#1}}

\newtheorem{theorem}{Theorem}[section]
\newtheorem{proposition}[theorem]{Proposition}

\newtheorem{lemma}[theorem]{Lemma}
\newtheorem{corollary}[theorem]{Corollary}
\newtheorem*{theorem*}{Theorem}
\newtheorem*{proposition*}{Proposition}

\theoremstyle{definition}

\newtheorem{definition}[theorem]{Definition}

\newtheorem{remark}[theorem]{Remark}
\newtheorem{assumption}[theorem]{Assumption}

\mathtoolsset{showonlyrefs}

\usepackage{import}
\usepackage{xifthen}
\usepackage{pdfpages}
\usepackage{transparent}

\begin{document}

\title{Uniform Convergence of Adversarially Robust Classifiers}
\author{Rachel Morris\thanks{North Carolina State University, \email{rachel.morris@mail.concordia.ca}} \and Ryan Murray\thanks{North Carolina State University, \email{rwmurray@ncsu.edu}}}

\maketitle

\begin{abstract}
In recent years there has been significant interest in the effect of different types of adversarial perturbations in data classification problems. Many of these models incorporate the adversarial power, which is an important parameter with an associated trade-off between accuracy and robustness. This work considers a general framework for adversarially-perturbed classification problems, in a large data or population-level limit. In such a regime, we demonstrate that as adversarial strength goes to zero that optimal classifiers converge to the Bayes classifier in the Hausdorff distance. This significantly strengthens previous results, which generally focus on $L^1$-type convergence. The main argument relies upon direct geometric comparisons and is inspired by techniques from geometric measure theory. 

\end{abstract}


\section{Introduction}

In recent years, neural networks have achieved remarkable success in a wide range of classification and learning tasks. However, it is now well-known that these networks do not learn in the same ways as humans and will fail in specific settings. In particular, a wide range of recent work has shown that they fail to be robust to specially designed adversarial attacks  \cite{engstrom2019spatial}, \cite{eykholt2018realAttack}, \cite{Goodfellow2014ExplainingAH},   \cite{qin2019imperceptible}, \cite{szegedy2013}. 

One general approach for mitigating this problem is to include an adversary in the training process. A simple mathematical formulation of this method for 0-1 loss \cite{madry2018}, in the large-data or population limit, is to consider the optimization problem
\[
\min_A J_\e(A), \qquad J_\e(A) :=  \mathbb{E} \left[\max_{\tilde x \in B(x,\e)} |\mathds{1}_A(\tilde x) - y|\right],\]
where the $y$ variables represent observed classification labels and $x$ variables represent features (we give a more precise description of our setting in the next section). This can be seen as a robust optimization problem, where an adversary is allowed to modify the inputs to our classifier up to some distance $\e$. When $\e =0$ this corresponds to the standard Bayes risk. 

Recent work has significantly expanded our mathematical understanding of this problem. Our work directly builds upon \textcite{bungert2023geom}, which  rewrites the previous functional as
\[
J_\e(A) =  \mathbb{E}[|\mathds{1}_A(x) - y|] + \e \Per_{\varepsilon}(A),
\]
where $\Per_{\varepsilon}$ is a special data-adapted perimeter, whose definition is given in \eqref{eps perim}. This is related to a growing body of recent work, for example showing that $\Per_{\varepsilon}$ converges to the (weighted) classical perimeter \cite{bungert2024gammaconv}, and demonstrating links between the adversarially robust training problem and mean curvature flow \cite{bungert2024mean}, \cite{ngt2022NecCond}. This literature seeks to provide a more complete description of the effect of $\e$ on adversarially robust classifiers in a geometric sense. This relates to the study of nonlocal perimeter minimization and flows   \cite{cesaroni2018minimizersnonlocal}, \cite{CesaroniNovaga2017isoperi}, \cite{chambolle2015nonlocalflows}, where the \textit{unweighted} $\e$-perimeter is considered. As training these robust classifiers is generally a challenging task, one overarching goal of this type of work is to provide a more precise understanding of the effect of $\e$, practical means for approximating that effect, and the impact on classifier complexity: each of these has the potential to improve more efficient solvers for these problems.

Various modifications of the robust classification energy $J_\e$ have been proposed. For example, some authors relax either the criteria for an adversarial attack or the loss function to interpolate between the accurate yet brittle Bayes classifier and the robust yet costly minimizers of the adversarial training problem \cite{bungert2023begins}, \cite{heredia2023random}, \cite{raman2023proper}, \cite{robey22a}. Still others employ optimal transport techniques to study distributionally robust optimization where instead of perturbing data points, the adversary perturbs the underlying data distribution \cite{frankNW2024minmax}, \cite{trillos2023existence},  \cite{trillosJacobsKimMOT1},  \cite{pydiJogOT}, \cite{pydiJog2024}.  

The main goal of this paper is to study the convergence of solutions to the adversarially robust classification problem towards the original Bayes classification task for data-perturbing models. We build a framework that allows us to consider a wide range of adversarial settings at the same time. In doing so, we obtain Hausdorff convergence results, which are generally much stronger than the $L^1$-type results previously obtained \cite{bungert2023begins}. These results parallel many of the basic results in the study of variational problems involving perimeters, wherein one first proves stability in $L^\infty$ spaces, and then subsequently proves stronger regularity results for minimizers. In a similar way, we see our results as a building block towards stronger regularity results for the adversarially robust classification problem, which have received significant attention in the literature. We begin by concretely describing the setup of our problem and then giving an informal statement of our results along with some discussion.

\subsection{Setup}\label{sec:setup}

Let the Euclidean space $\R^d$ equipped with the metric $\dist(\cdot,\cdot)$ represent the space of features for a data point, and let $\cB(\R^d)$ be the set of all Borel measurable subsets of $\R^d$. We will let $\mathcal L^d$ be the $d$-dimensional Lebesgue measure. We are considering a supervised binary classification setting, in which training pairs $(x,y)$ are distributed according to a probability measure $\mu$ over $\R^d\times \{0,1\}$. Here $y$ represents the class associated with a given data point, and the fact that $y \in \{0,1\}$ corresponds to the binary classification setting. Let $\rho$ denote the $\R^d$ marginal of $\mu$, namely $\rho(A) = \mu(A\times \{0,1\})$. We decompose $\rho \in \cP(\R^d)$ into $\rho = w_0\rho_0 + w_1\rho_1$ where $w_i = \mu(\R^d \times \{i\})$, and the conditional probability measure $\rho_i \in \cP(\R^d)$ for a set $A\in\cB(\R^d)$ is \[\rho_i(A) = \frac{\mu(A\times \{i\})}{w_i}\]for $i = 0,1$. All of these measures are assumed to be Radon measures. 

In binary classification, we associate a set $A\in\cB(\R^d)$ with a classifier, meaning that $x\in \R^d$ is assigned label 1 when $x\in A$ and $x$ is assigned the label 0 when $x\in A^\comp$. Unless otherwise stated, we will assume all classifiers $A\in \cB(\R^d).$ The Bayes classification problem for the 0-1 loss function is given by
\begin{equation}
\inf_{A\in \cB(\R^d)}\E_{(x,y)\sim \mu}[|\mathds{1}_A(x) - y|].
\label{eqn:BC}
\end{equation} 
In this work, we will only consider the 0-1 loss function
, which allows us to restrict our attention to indicator functions for minimizers of \eqref{eqn:BC}. 
We refer to minimizes of the Bayes risk as \textit{Bayes classifiers}. 
\begin{remark}[Uniqueness of Bayes Classifiers]
If we assume that $\rho$ has a density everywhere on $\R^d$ and  identify the measures $\rho_i$ with the density at $x\in \R^d$ given by $\rho_i(x)$, then we can describe the uniqueness, or lack thereof, of Bayes classifiers in terms of those densities. Specifically, Bayes classifiers are unique up to the set $\{w_1\rho_1 = w_0\rho_0\}$, which may be a set of positive measure depending on $\mu$. We define maximal and minimal Bayes classifiers (in the sense of set inclusion) by 
\begin{equation}A_0^{\max} \coloneqq \{x\in \R^d: w_1\rho_1(x) \ge w_0\rho_0(x)\}, \qquad A_0^{\min} \coloneqq \{x\in \R^d: w_1\rho_1(x) > w_0\rho_0(x)\}. \label{minmax BC}\end{equation} 

When $w_0\rho_0-w_1\rho_1 \in C^1$ and $|w_0\grad \rho_0 - w_1\grad \rho_1| > \alpha > 0$ on the set $\{w_0\rho_0 = w_1\rho_1\}$, the Bayes classifier is unique up to sets of $\rho$ measure zero. In the case where $\rho$ is supported everywhere, Bayes classifiers are unique up to sets of $\mathcal L^d$ measure zero. Whenever we refer to the Bayes classifier as unique, we mean unique in this measure-theoretic sense. Later on in Assumption~\ref{nondegen assumption}, we will refer to such uniqueness as the ``non-degeneracy'' of the Bayes classifier and represent the unique classifier by $A_0$. 
\label{rem:BC uniqueness}
\end{remark}

Throughout this paper, we will consider and seek to unify two optimization problems from the literature that aim to train robust classifiers. First, we consider the \textit{adversarial training problem}, which trains classifiers to mitigate the effect of worst-case perturbations \cite{madry2018}. The adversarial training problem is 
\begin{equation}
    \inf_{A\in \cB(\R^d)}\E_{(x,y)\sim \mu}\Bigg[\sup_{\tx\in B_\dist(x,\varepsilon)}|\mathds{1}_A(\tx) - y|\Bigg],
\end{equation}
where $B_\dist(x,\varepsilon)$ is the \textit{open} metric ball of radius $\varepsilon> 0$. The existence of solutions in this setting was previously established \cite{bungert2023geom}. The parameter $\varepsilon$ is called the \textit{adversarial budget}, and it represents the strength of the adversary. 
By using the open ball, we are following the conventions set in the previous work on convergence of optimal adversarial classifiers \cite{bungert2024gammaconv}. Other works have utilized the closed ball due to consistency with the standard classification problem when $\eps = 0$, but that comes at the price of added measurability concerns: see Remark~\ref{rem:atp related work} for more details. 

An equivalent form of this variational problem  (see \cite{bungert2023geom}), wherein the problem is rewritten using a nonlocal perimeter, is
\begin{equation}\inf_{A\in \cB(\R^d)} \E_{(x,y)\sim\mu}[|\mathds{1}_A(x) - y|] + \varepsilon\Per_{\varepsilon}(A)     \label{eqn:ATP}\end{equation} with the \textit{$\varepsilon$-perimeter} defined by \[\Per_{\varepsilon}(A) \ceq \frac{w_0}{\varepsilon}\int_{\R^d} \left[\sup _{\tx\in B_\dist(x,\varepsilon)}\mathds{1}_A(\tx) - \mathds{1}_A(x)\right] \, d\rho_0(x) + \frac{w_1}{\varepsilon}\int_{\R^d} \left[\mathds{1}_A(x)-\inf _{\tx\in B_\dist(x,\varepsilon)}\mathds{1}_A(\tx) \right]\, d\rho_1(x).\] 
This normalization with $\varepsilon$ in the denominator is chosen so that we recover the (weighted) classical perimeter as $\varepsilon\to 0^+$. In this sense, we consider the nonlocal $\varepsilon$-perimeter a data-adapted approximation of the classical perimeter. From the variational problem given by \eqref{eqn:ATP}, we define the \textit{adversarial classification risk} for a classifier $A\in \cB(\R^d)$ as \[J_\varepsilon(A) \ceq \E_{(x,y)\sim \mu}[|\mathds{1}_A(x)-y|] + \e\Per_{\varepsilon}(A).\]

When considering the $\varepsilon$-perimeter, the region affected by adversarial perturbations must be within distance $\varepsilon$ of the decision boundary of the classifier. As such, it will be helpful to be able to discuss sets that either include or exclude the $\varepsilon$-perimeter region. From mathematical morphology \cite{serra83morph}, for a set $A\in\R^d$ and $\varepsilon>0$ we define the
\begin{itemize} 
\item \textit{$\varepsilon$-dilation} of $A$ as $A^\varepsilon\ceq \{x\in \R^d: \dist(x,A) <\varepsilon\}$, 
\item \textit{$\varepsilon$-erosion} of $A$ as $A^{-\varepsilon}\ceq \{x\in \R^d: \dist(x,A^\comp) \ge \varepsilon\}$.
\end{itemize}
Using this notation, one can equivalently express the $\varepsilon$-perimeter as \begin{equation}\text{Per}_\varepsilon(A) \ds = \frac{w_0}{\varepsilon}\rho_0(A^\varepsilon\setminus A)+ \frac{w_1}{\varepsilon}\rho_1(A\setminus A^{-\varepsilon}).\label{eps perim}\end{equation} 
Inspired by the notation in geometric measure theory, we also define the \textit{relative $\varepsilon$-perimeter} for a classifier $A\in \cB(\R^d)$ with respect to a set $E\in\cB(\R^d)$ by
\begin{equation}\Per_{\varepsilon}(A;E) \ceq \frac{w_0}{\varepsilon}\rho_0((A^\varepsilon\setminus A) \cap E)+ \frac{w_1}{\varepsilon}\rho_1((A\setminus A^{-\varepsilon})\cap E).\label{rel eps perim}\end{equation}


\begin{remark}[Previous work for the adversarial training problem \eqref{eqn:ATP}] 

The worst-case adversarial training model was initially proposed for general loss functions by \textcite{madry2018}. When the loss function is specified to be the 0-1 loss function, previous work has established the existence and considered the equivalence of minimizers to \eqref{eqn:ATP} for the open and closed ball models \cite{Awasthi2021OnTE}, \cite{bungert2023geom},  \cite{trillos2023existence}, \cite{pydiJog2024}. Although the open and closed ball models are similar, there are some subtle differences that must be considered. While measurability of $\sup_{\tx \in B_\dist(x,\e)} \mathds 1_A(\tx)$ for a Borel set $A$ in the open ball model is trivial, the same cannot be said for the closed ball model; to address these measurability concerns in the closed ball model, one must employ the universal $\sigma$-algebra instead of the Borel $\sigma$-algebra. We emphasize that we choose to study the open ball model as this simplifies the analysis and measurability concerns associated with the closed ball model, and the open ball model was used for prior convergence results \cite{bungert2024gammaconv}. 

Some papers consider a surrogate adversarial risk which is more computationally tractable \cite{Awasthi2021calibration}, \cite{Bao2020calibrated},     \cite{frankNW2024minmax}, \cite{Meunier2022consistency}; others explore necessary conditions and geometric properties of minimizers \cite{bungert2024mean}, \cite{bungert2024gammaconv}, \cite{ngt2022NecCond}. Of particular note to the present work is the study of the limit of minimizers of $J_\e$. Theorem 2.5 states \cite{bungert2024gammaconv} ,
\begin{theorem*}[Conditional convergence of adversarial training]
    Under the conditions of Theorems 2.1 and 2.3 from \cite{bungert2024gammaconv} and assuming the source condition, any sequence of solutions to \[\inf_{A\in \mathcal B(\Omega)}\E_{(x,y)\sim\mu}\left[\sup_{\tilde x \in B(x,\e)\cap \Omega} |\mathds{1}_A(\tilde x) - y|\right]\] possesses a subsequence converging to a minimizer of \[\min\{ \Per(A;\rho):A\in \argmin_{B\in \mathcal B(\Omega)} \E_{(x,y)\sim \mu}[|\mathds{1}_B(x)-y|]\}.\]
    \label{thm:gamma conv}
\end{theorem*} 
The convergence is proven in the $L^1(\Omega)$ topology for some open, bounded Lipschitz domain $\Omega\subset \R^d$. Here, $\Per(\cdot;\rho)$ is a weighted version of the classical perimeter. The source condition mentioned provides minor regularity assumptions on the Bayes classifier. Note that in the referenced theorem, there are additional assumptions on the underlying data distribution $\rho$. In our work, we strengthen this convergence result by proving Hausdorff convergence of minimizers of \eqref{eqn:ATP} to the Bayes classifier with similar assumptions on $\rho$. 
\label{rem:atp related work}
\end{remark}


The second optimization problem which serves as an important model case interpolates between the accuracy on clean data of the Bayes classifier and the robustness of the adversarial training problem minimizers. The \textit{probabilistic adversarial training problem} for $p\in [0,1)$ and probability measures $\mathfrak p_x\in \mathcal P(\R^d)$ for each $x\in \R^d$ is 
\begin{equation} \inf_{A\in\cB(\R^d)} \E_{(x,y)\sim\mu}[|\mathds{1}_A(x)-y|]+\text{ProbPer}_p(A), \label{PATP} \end{equation} with the \textit{probabilistic perimeter} defined by
\begin{equation} 
\text{ProbPer}_p(A) \ceq w_0\rho_0(\Lambda_p^0(A)) + w_1\rho_1(\Lambda_p^1(A)), \label{prob perim} 
\end{equation} and the set functions $\Lambda_p^i$ for $i = 0,1$ defined by 
\begin{align}\Lambda_p^0(A) &\ceq \{x\in A^\comp:  \, \prob(x'\in A: x'\sim \mathfrak p_x)>p\},\\ 
\Lambda_p^1(A) &\ceq \{x\in A: \,\prob(x'\in A^\comp: x'\sim \mathfrak p_x)>p\}.\label{prob perim sets}\end{align} Here, $\prob(x'\in A: x'\sim \mathfrak p_x)$ is the probability that a point $x'$ sampled from the probability distribution $\mathfrak p_x$ belongs to the set $A$. We notice that \eqref{prob perim} takes the same form as \eqref{eps perim} where we replace the metric boundary fattening by a probabilistic fattening. We define the \textit{probabilistic adversarial classification risk} for a classifier $A\in\cB(\R^d)$ as \[J_p(A)\ceq \E_{(x,y)\sim\mu}[|\mathds{1}_A(x)-y|] + \ProbPer_{p}(A).\]
 The \textit{relative probabilistic perimeter} for a classifier $A\in \cB(\R^d)$ with respect to a set $E\in\cB(\R^d)$ is given by \begin{equation}\text{ProbPer}_p(A;E) \ceq w_0\rho_0(\Lambda_p^0(A)\cap E) + w_1\rho_1(\Lambda_p^1(A)\cap E).\label{rel prob perim}\end{equation} 

To make the connection with the $\varepsilon$-perimeter more concrete, we will restrict our attention to certain families of probability measures that scale appropriately with $\varepsilon$ for the remainder of this work. 

\begin{assumption}
Let $\xi: \R^d \to [0,\infty)$ such that $\xi\in L^1(\R^d)$, $\int_{\R^d}\xi(z) \, dz = 1$, $\xi(z) = 0$ if $|z| > 1$, and $\xi(z) > c$ for some constant $c>0$ and for $|z| \leq 1$. For $x,x' \in \R^d$, we assume that \[\mathfrak p_{x,\varepsilon}(x') = \varepsilon^{-d}\xi\left(\frac{x'-x}{\varepsilon}\right).\]
\label{p assumption}
\end{assumption} 

We will now write ProbPer$_{\varepsilon,p}$ and refer to it as the probabilistic $\varepsilon$-perimeter to emphasize the dependence on the adversarial budget.  Unlike with the Per$_\varepsilon$, we do not normalize ProbPer$_{\varepsilon,p}$ with respect to $\varepsilon$. We also write $J_{\e,p}$ instead of $J_p$ and $\Lambda_{\e,p}^i$ instead of $\Lambda_{p}^i$ for $i=0,1$. Under Assumption~\ref{p assumption}, $\Lambda_{\e,p}^0(A)$ and $\Lambda_{\e,p}^1(A)$ are subsets of the $\varepsilon$-perimeter regions $A^\varepsilon\setminus A$ and $A\setminus A^{-\varepsilon}$, respectively. Specifically, this means that $J_{\e,p}(A)\le J_{\e}(A)$ for all $A\in \cB(\R^d)$ when the underlying data distribution $\mu$ is the same. We note that probabilistic $\e$-perimeter that most closely coincides with the $\e$-perimeter when $p = 0$ and $\mathfrak p_{x,\varepsilon} = \text{Unif}(B_\dist(x,\e))$ for each $x\in \R^d$. 

\begin{remark}[Previous work for the probabilistic adversarial training problem \eqref{PATP}]

This form of the problem was proposed by \textcite{bungert2023begins} as a revision of  probabilistically robust learning   \cite{robey22a}. Although ProbPer$_p$ is not a perimeter in the sense that it had not been shown to be submodular and it does not admit a coarea formula, we follow the convention from \cite{bungert2023begins} and refer to ProbPer$_p$ as the probabilistic perimeter. Importantly, existence of minimizers has not been proved for either the original or modified probabilistic adversarial training problem. There have also been no results pertaining to the convergence of minimizers, provided they exist, to the Bayes classifier for either version. 

However, \textcite{bungert2023begins} propose and prove the existence of minimizers for a related probabilistically robust $\Psi$ risk
\begin{equation*}J_\Psi(A) \ceq \E_{(x,y)\sim \mu} [|\mathds{1}_A(x) - y|] + \ProbPer_{\Psi}(A) \end{equation*} for suitable functions $\Psi:[0,1]\to [0,1]$ where the $\Psi$-perimeter takes the form \begin{equation}\ProbPer_{\Psi}(A)\ceq \int_{A^c} \Psi(\prob(x'\in A:x'\sim \mathfrak p_{x})) d\rho_0(x) + \int_{A} \Psi(\prob(x'\in A^\comp:x'\sim \mathfrak p_{x})) d\rho_1(x). \label{eqn:psi per}\end{equation} However, the convergence results proved in this paper do \textit{not} currently extend to the $\Psi$-perimeter case. The details will be further discussed in Remark~\ref{rem:psi per}.  
\label{rem: prev patp}
\end{remark}

If we juxtapose the variational problem for the adversarial training problem \eqref{eqn:ATP} and the probabilistic adversarial training problem \eqref{PATP}, both risks are of the form
\[ J(A) = \text{ Bayes risk } + \text{ data-adapted perimeter, }\] where the data-adapted perimeters can be expressed as \[\text{data-adapted perimeter} = w_0\rho_0(\text{subset of } A^\comp) + w_1\rho_1(\text{subset of } A).\] We seek to develop a unifying framework for various adversarial models including, but not limited to, \eqref{eqn:ATP} and \eqref{PATP}. These types of attacks are designed to flexibly capture a range of adversarial behaviors, not just the idealized ones given in the original adversarial training problem. Under the proper assumptions, which will be discussed in  Sections~\ref{sec:framework} and \ref{sec:det conv}, we can extend the convergence result to a broad class of adversarial attacks. We begin by giving some concrete definitions. 

\begin{definition}
    For a classifier $A\in \cB(\R^d)$, we define the Lebesgue measurable function $\phi:\R^d \to \{0,1\}$ by
\[\phi(x;A) \coloneqq \begin{cases}
     1, & \text{if the adversary can perturb a data point } x \text{ from } A \text{ to } A^\comp \text{ or vice versa},\\
     0, & \text{otherwise}.
\end{cases}\]  We refer to $\phi$ as the \textit{deterministic attack function} with respect to the classifier $A$. 

\end{definition}

Deterministic refers to the fact that the classification risk is completely determined at any point $x\in\R^d$ by the choice of classifier and the associated attack function. We emphasize that this attack function does not consider the true label $y$ associated with $x$. 

In order to generalize the classification risk, it will be essential to isolate the sets where classification loss occurs. We can define the following set operators based on the values of $\phi$.
\begin{definition} 

    Let $A\in \cB(\R^d)$. For a deterministic attack function $\phi$, we define the set operators $\Lambda_{\phi}^i:\cB(\R^d)\to \cB(\R^d)$ and $ \tilde\Lambda^i_\phi:\cB(\R^d)\to \cB(\R^d)$ 
    for $i = 0,1$ by
    \[\Lambda^0_{\phi}(A) \ceq \{x\in A^\comp: \phi(x; A) = 1\},
     \qquad \Lambda^1_\phi(A) \ceq \{x\in A: \phi(x; A) = 1\},  \] 
    \[ \tilde\Lambda^0_\phi(A) \ceq \{x\in A^\comp: \phi(x; A) = 0\}, \qquad \tilde\Lambda^1_\phi(A) \ceq \{x\in A: \phi(x; A) = 0\}.  \]
    We refer to these four sets collectively as \textit{$\Lambda$-sets}. For convenience, we also define $\Lambda_\phi(A) = \Lambda_\phi^0(A)\cup \Lambda_\phi^1(A)$ and $\tilde{\Lambda}_\phi(A) = \tilde\Lambda_\phi^0(A)\cup \tilde\Lambda_\phi^1(A)$. Note the $0$ and $1$ superscripts indicate the label assigned by the classifier $A$ and \textit{not} the value of the deterministic attack function (i.e. $0$ corresponds to points in $A^\comp$ and $1$ corresponds to points in $A$).
    \label{def:lambda sets}
\end{definition}

The set $\Lambda_\phi(A)$ contains points that meet the attack criteria for the deterministic attack function $\phi$ whereas the set $\tilde{\Lambda}_\phi(A)$ contains points that do not meet the attack criteria. The $\Lambda$-sets are mutually disjoint with $A = \Lambda^1_\phi(A) \cup \tilde\Lambda^1_\phi(A)$ and $A^\comp = \Lambda^0_{\phi}(A) \cup \tilde\Lambda^0_\phi(A)$. 

We can express the classification risk for a set $A\in \cB(\R^d)$ by the loss on the attacked sets, given by $\Lambda_\phi(A)$, and by the loss inherent to the choice of classifier. More formally, we define the generalized classification risk as follows.

\begin{definition}
    The \textit{generalized classification risk} for a deterministic attack function $\phi$ and classifier $A\in \cB(\R^d)$ is given by 
    \begin{equation}J_\phi(A) \ceq w_0\rho_0(\Lambda_\phi^0(A) \cup A) + w_1\rho_1(\Lambda_\phi^1(A)\cup A^\comp). \label{Jphi}\end{equation}
\end{definition}

As in \cite{TRADESpaper}, we seek to separate the total classification risk $J_\phi$ into the standard Bayes risk (natural error) and the risk attributed to the adversary's attack.

\begin{definition}
    The \textit{adversarial deficit} for a classifier $A\in \cB(\R^d)$ and a deterministic attack function $\phi$ is defined to be \[D_\phi(A) \ceq J_\phi(A) - \E_{(x,y)\sim \mu}[|\mathds{1}_A(x) - y|],\] where $\E_{(x,y)\sim \mu}[|\mathds{1}_A(x) - y|]$ is the standard Bayes risk.
\end{definition}

As one can express the standard Bayes risk as \[\E_{(x,y)\sim \mu}[|\mathds{1}_A(x) - y|] = w_0\rho_0(A) + w_1\rho_1(A^\comp),\] we can derive a more useful equation for the adversarial deficit that mirrors the formulas for the data-adapted perimeters \eqref{eps perim} and \eqref{prob perim}, namely, 
\begin{equation}
    D_\phi(A) = w_0\rho_0(\Lambda^0_\phi(A))+ w_1\rho_1(\Lambda^1_\phi(A)). \label{eqn:rel adv def}
\end{equation}
Unlike the data-adapted perimeters we described above, at this stage $\Lambda_\phi(A)$ is \textit{not} necessarily in some neighborhood of the decision boundary. We define the relative adversarial deficit for a classifier $A\in\cB(\R^d)$ with respect to a set $E\in\cB(\R^d)$ to be \[ D_{\phi}(A;E) \ceq  w_0\rho_0(\Lambda^0_\phi(A)\cap E)+w_1\rho_1(\Lambda^1_\phi(A)\cap E). \]

With the appropriate definitions in place, we now present the generalized adversarial training problem for the deterministic attack function $\phi$. 

\begin{definition}
    For a deterministic attack function $\phi$, the \textit{generalized adversarial training problem} is given by \begin{equation}\inf_{A\in \cB(\R^d)} \E_{(x,y)\sim \mu}[|\mathds{1}_A(x) - y|] + D_\phi(A). \label{DATP}\end{equation}
\end{definition}

In the previous equation, the adversarial deficit, $D_\phi$, takes the place of the data-adapted perimeter terms from \eqref{eqn:ATP} and \eqref{PATP}. 

\begin{remark}
    By construction, the adversarial training problem \eqref{eqn:ATP} and the probabilistic adversarial training problem \eqref{PATP} are two examples that fall under this generalized attack function framework. For \eqref{eqn:ATP}, the \textit{$\e$-deterministic attack function} with respect to a classifier $A\in\cB(\R^d)$ for $\e > 0$ is
    \[\phi_\e(x;A)\ceq \begin{cases}
    1, & \text{if } \dist(x,\partial A) <\varepsilon,\\
    0, & \text{otherwise}.
\end{cases}\] For $\phi_\e$, we will let $\Lambda_\e^0(A) \ceq A^\e\setminus A$, $\Lambda_\e^1(A) \ceq A\setminus A^{-\e}$, $\tilde\Lambda^0_\e(A) \ceq A^\comp \setminus A^{\e}$, and  $\tilde\Lambda^1_\e(A) \ceq A^{-\e}$ denote the $\Lambda$-sets for convenience.

On the other hand for \eqref{PATP}, the \textit{$(\e,p)$-deterministic attack function} with respect to a classifier $A\in\cB(\R^d)$ for $\e > 0$ and $p\in[0,1)$ is \[ \phi_{\varepsilon, p}(x;A) = \begin{cases}
    1, & \text{if }\prob(\mathds{1}_A(x') \neq \mathds{1}_A(x)): x'\sim\mathfrak p_{x,\varepsilon}) > p,\\
    0, & \text{otherwise}.
\end{cases}\]
\label{rem: spec attack fns}
\end{remark}

\subsection{Informal Main Results and Discussion}\label{inf results}

We will focus the main results and discussion on the generalized adversarial training problem \eqref{DATP} and comment on the application to the adversarial training problem \eqref{eqn:ATP} and the probabilistic adversarial training problem \eqref{PATP} when appropriate. By Remark~\ref{rem: spec attack fns}, all statements pertaining to \eqref{DATP} automatically apply to \eqref{eqn:ATP} and \eqref{PATP}. However because \eqref{eqn:ATP} is sensitive to measure zero changes,  results for \eqref{eqn:ATP} are stronger than what can be stated in the generalized or probabilistic cases. On the other hand, the results for \eqref{PATP} are identical to those for \eqref{DATP} up to notation. 

The first crucial result for \eqref{DATP} provides an estimate on the relative adversarial deficit. 

\begin{proposition*} [(Informal) Energy Exchange Inequality for \eqref{DATP}]
 Under mild assumptions on $\phi$ (see Assumption~\ref{phi consistency}), for a classifier $A\in \cB(\R^d)$ and a set $E\in \cB(\R^d)$ such that $w_0\rho_0 - w_1\rho_1 > \delta > 0$ on $E$, if $J_\phi(A\setminus E) - J_\phi(A) \ge 0$, then \[D_{\phi}(A;E) \le D_{\phi}(E^\comp;A) - \delta\mathcal L^d(A\cap E) + w_0\rho_0(\widehat U_{11})+w_1\rho_1(\widehat U_{1})\label{eei inequal}\] where 
 $\widehat U_{1} \subset \tilde\Lambda_\phi^1(A) \cap \tilde\Lambda_\phi^0(E)$ and $\widehat U_{11} \subset {\Lambda_\phi^0}(A) \cap {\Lambda_\phi^1}(E)$. \label{prop: eei informal}\end{proposition*}

 \begin{figure}[h!]
    \centering
    \scalebox{.75}{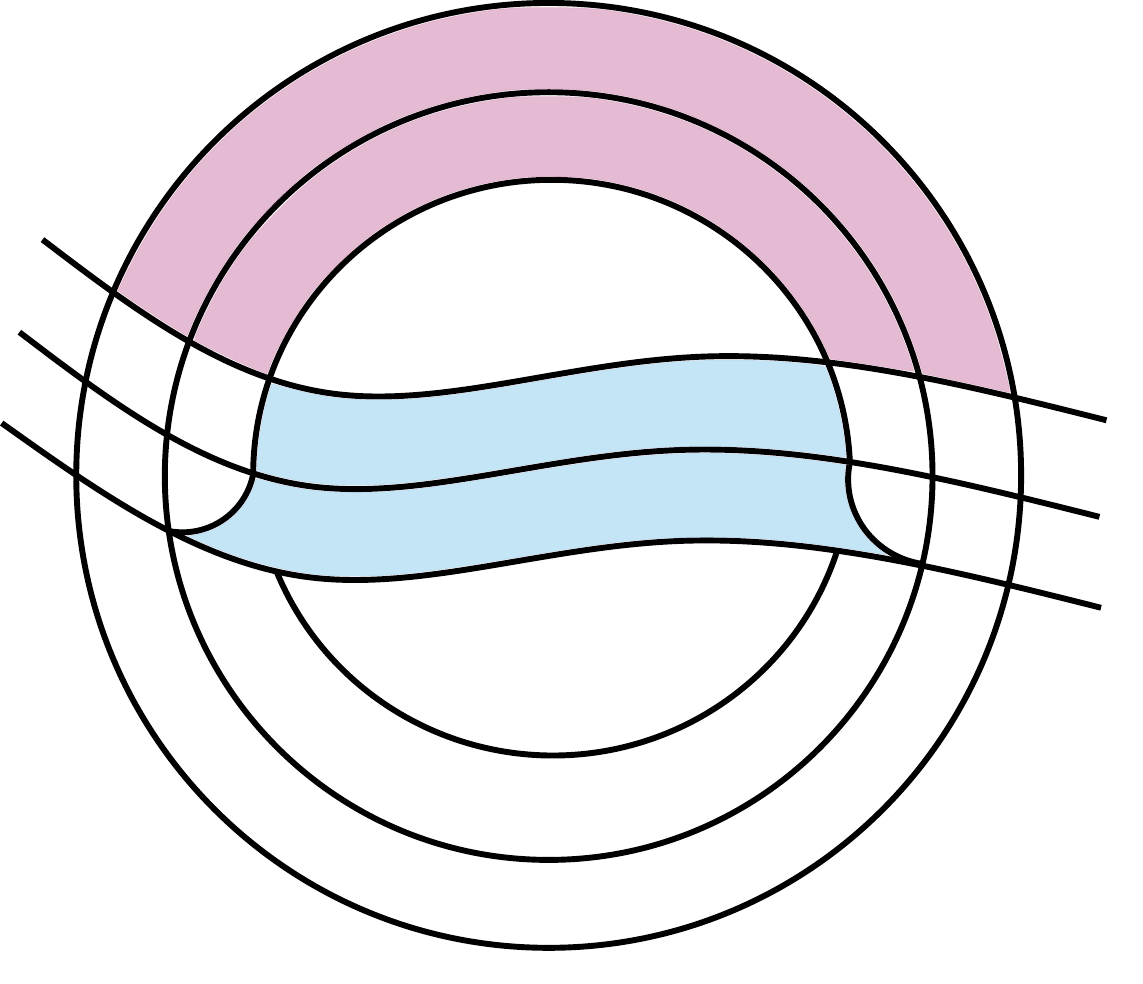}
    \caption{This diagram illustrates the sets present in the energy exchange inequality for the adversarial training problem \eqref{eqn:ATP} when $E = B_\dist(R)$. The sets comprising $\e\Per_{\varepsilon}(A;B_\dist(R))$ are shaded blue and purple whereas the sets comprising $\e\Per_{\varepsilon}(B_\dist(R)^\comp;A)$ are shaded pink and purple.}
    \label{fig:eperEEI}
\end{figure}

The energy exchange inequality asserts that, if it favorable according to the densities
to be labeled 0 on $E$ but adversarial training labels it 1, then the ``perimeter'' (more generally, the adversarial deficit) of the
original set $A$ must be quantifiably better in the sense of \eqref{eei inequal}. In spirit, the energy exchange inequality is connected to relative isoperimetric comparisons as it seeks to relate the relative adversarial deficits (or for \eqref{eqn:ATP} the relative $\e$-perimeters) of two sets to the volume of their intersection. However, the energy exchange inequality has additional error terms that must be accounted for. In the case of the stronger $\e$-perimeter, $\widehat U_{1} = \emptyset$ so the energy exchange inequality simplifies and can be expressed as follows.

\begin{proposition*} [(Informal) Energy Exchange Inequality for \eqref{eqn:ATP}]
For a classifier $A\in \cB(\R^d)$ and a set $E\in \cB(\R^d)$ such that $w_0\rho_0 - w_1\rho_1 > \delta > 0$ on $E$, if $J_\e(A\setminus E) - J_\e(A) \ge 0$, then \[\e\Per_\e(A;E) \le \e\Per_{\e}(E^\comp;A) - \delta\mathcal L^d(A\cap E) + w_0\rho_0(\widehat U_{11})\] where $\widehat U_{11} \subset  (A^\varepsilon\setminus A) \cap (E\setminus E^{-\varepsilon})$ (see Figure~\ref{fig:eperEEI}). \end{proposition*}

As for the relative probabilistic perimeter $\ProbPer_{\e,p}$, the energy exchange inequality is the same as that for \eqref{DATP} up to notation.

\begin{proposition*} [(Informal) Energy Exchange Inequality for \eqref{PATP}]
 For a classifier $A\in \cB(\R^d)$ and a set $E\in \cB(\R^d)$ such that $w_0\rho_0 - w_1\rho_1 > \delta > 0$ on $E$, if $J_{\e,p}(A\setminus E) - J_{\e,p}(A) \ge 0$, then \[\ProbPer_{\e,p}(A;E) \le \ProbPer_{\e,p}(E^\comp;A) - \delta\mathcal L^d(A\cap E) + w_0\rho_0(\widehat U_{11})+w_1\rho_1(\widehat U_{1})\] where $\widehat U_{1} \subset (A\setminus A^{-\e})\cap (E^\e \setminus E)$ and $\widehat U_{11} \subset  (A^\varepsilon\setminus A) \cap (E\setminus E^{-\varepsilon})$. \end{proposition*}

The energy exchange inequality allows us to argue that classifiers which are minimizers of the generalized adversarial training problem \eqref{DATP}, if they exist, can be made disjoint from sets where it is energetically preferable to be labeled 0 when the adversarial budget $\e$ is small enough. As we will see, in the generalized setting we can only guarantee the uniqueness of minimizers of \eqref{PATP} and \eqref{DATP} up to sets of measure zero; however, we can show that the intersection of such sets with an energetic preference for the label zero with minimizers must have $\mathcal L^d$ measure zero. For the adversarial training problem \eqref{eqn:ATP}, we can improve the result to show that any minimizer must be disjoint from these sets when $\e$ is small enough. This result builds towards proving uniform convergence of minimizers of \eqref{DATP} to the Bayes classifier, which is the next main result. In order to prove the convergence rate, we must include a non-degeneracy assumption to ensure $\dist_H(A_0^{\max},A_0^{\min}) = 0$ and the Bayes classifier is unique in the sense of Remark~\ref{rem:BC uniqueness}.

\begin{theorem*} [Informal] With mild assumptions on $\phi$, let $K$ be compact and let $\{A_{\e,\phi}\}_{\e>0}$ be any sequence of minimizers to the generalized adversarial training problem \eqref{DATP}. Assuming that $w_0 \rho_0-w_1\rho_1$ is non-degenerate, then 
\[\dist_H((A_{\varepsilon,\phi}\cup N_1\setminus N_2) \cap K, A_0\cap K) \to 0\]
as $\e\to0^+$, where $N_1,N_2$ are sets of $\mathcal L^d$ measure zero, $\dist_H$ is the Hausdorff distance, and $A_0$ is the Bayes classifier. 
\end{theorem*}

However, the theorem actually proved is more general and does not require a unique Bayes classifier. Under these relaxed assumptions, we prove a corralling result for the sequence $\{A_{\varepsilon,\phi}\}_{\e > 0}$ with respect to the Hausdorff distance from the maximal Bayes classifier, $A_0^{\max}$, and the minimal Bayes classifier, $A_0^{\min}$. In essence, the corralling result states that the boundary of $\lim_{\e\to 0^+} A_{\varepsilon,\phi}\cup N_1\setminus N_2$ must lie between the boundaries of $A_0^{\max}$ and $A^{\min}_0$. When we specify this result to the adversarial training problem \eqref{eqn:ATP}, we no longer have to remove a $\mathcal L^d$ measure zero set and instead prove the following.

\begin{theorem*} [Informal] Let $K$ be compact and let $\{A_{\e}\}_{\e>0}$ be any sequence of minimizers to the  adversarial training problem \eqref{eqn:ATP}. Assuming that $w_0 \rho_0-w_1\rho_1$ is non-degenerate, then 
\[\dist_H(A_{\varepsilon} \cap K, A_0\cap K) \to 0\]
as $\e\to0^+$, where $\dist_H$ is the Hausdorff distance and $A_0$ is the Bayes classifier. 
\end{theorem*}

For the probabilistic adversarial training problem, the uniform convergence result states, 

\begin{theorem*} [Informal] Let $K$ be compact and let $\{A_{\e,p}\}_{\e>0}$ be any sequence of minimizers to the probabilistic adversarial training problem \eqref{PATP} for some fixed $p\in [0,1)$. Assuming that $w_0 \rho_0-w_1\rho_1$ is non-degenerate, then 
\[\dist_H((A_{\varepsilon,p}\cup N_1 \setminus N_2) \cap K, A_0\cap K) \to 0\]
as $\e\to0^+$, where $N_1,N_2$ are sets of $\mathcal L^d$ measure zero, $\dist_H$ is the Hausdorff distance, and $A_0$ is the Bayes classifier. 
\end{theorem*}

As with \eqref{DATP}, if we relax the assumption that the Bayes classifier is unique, we can instead prove an analogous corralling results with respect to $A_0^{\max}$ and $A_0^{\min}$ for \eqref{eqn:ATP} and \eqref{PATP}. 

With the non-degeneracy condition in place, we can also consider the rate of convergence and show that it is at most $O(\varepsilon^{\frac{1}{d+2}})$ for all three adversarial training problems. However, we do not expect this result to be optimal and would expect that the convergence rate to be $O(\e)$, which we discuss further in Remark~\ref{rem: conv rate}. 


\section{Energy Exchange Inequality}\label{sec:framework}

In this section, we will prove a quantitative result for the adversarial deficit, which can then be applied to the $\e$-perimeter and the probabilistic $\e$-perimeter. In order to do so, we will require the deterministic attack function $\phi$ and the corresponding $\Lambda$-sets to have the following structural properties.

\begin{assumption}
Recall Definition~\ref{def:lambda sets}. Let $A,E \in \mathcal B(\R^d)$. We will make the following two assumptions to ensure consistency with respect to complements and set difference:
\begin{enumerate}
\item Complement Property (CP): $\phi(x;A) = \phi(x;A^\comp)$, or in terms of $\Lambda$-sets,  $\Lambda_\phi^0(A) = \Lambda_\phi^1(A^\comp)$ and $\tilde\Lambda_\phi^0(A) = \tilde\Lambda_\phi^1(A^\comp)$. 
\item $\Lambda$-Monotonicity ($\Lambda$M): 
\begin{enumerate}
\item[(i)] If $x\in \tilde\Lambda_\phi^0(A)$, then $x\in \tilde\Lambda_\phi^0(A\setminus E)$.
\item[(ii)] If $x\in \tilde\Lambda_\phi^1(E)$, then $x\in \tilde\Lambda_\phi^0(A\setminus E)$.
\item[(iii)] If $x\in \Lambda_\phi^0(E)\cap A$, then $x\in \Lambda_\phi^1(A\setminus E)$.
\item[(iv)] If $x\in \Lambda_\phi^1(A)\cap E^\comp$, then $x\in \Lambda_\phi^1(A\setminus E)$. 
\end{enumerate}
\end{enumerate}
\label{phi consistency}
\end{assumption}

In the following series of remarks, we seek to better understand these two properties generally and as they apply to the adversarial and probabilistic adversarial settings. 

\begin{remark}[On Monotonicity]
    We note that the deterministic attack functions $\phi$ that satisfy Assumption~\ref{phi consistency} are \textit{not} monotonic with respect to set inclusion unless $\phi$ is the trivial attack function (i.e. $\phi \equiv 0$ or $\phi \equiv 1$). To illustrate this, suppose $\phi$ is monotonic. By the monotonicity of $\phi$ with respect to set inclusion coupled with the complement property,
    \[
    \phi(x;A) \stackrel{\rm (CP)}{=} \phi(x;A^\comp)\le \phi(x;(A\setminus E)^\comp)\stackrel{\rm (CP)}{=} \phi(x;A\setminus E)\le \phi(x;A).
    \]
    This implies $\phi(x;A) \equiv \phi(x;A\setminus E)$ and, if we let $E = A$, that $\phi(x;A)\equiv \phi(x;\emptyset)$. Hence, the attack is independent of $A$, which can only be satisfied by a trivial attack function. 

    Although $\phi$ itself is not monotonic, if you have a function $\psi$ which is monotonic in terms of set inclusion then setting $\phi$ via its level set yields an attack function which satisfies $\Lambda$-monotonicity. In particular, both the distance function and the probability function are monotonic. 
\end{remark}

\begin{remark}
We will verify that the adversarial training problem \eqref{eqn:ATP} and the probabilistic adversarial training problem \eqref{PATP} satisfy Assumption~\ref{phi consistency}. Recall from Remark~\ref{rem: spec attack fns}, the attack for \eqref{eqn:ATP} is denoted $\phi_\e$ and the attack for \eqref{PATP} is denoted $\phi_{\e,p}$  for some $\e > 0$ and $p\in [0,1)$.

We will first show that $\phi_\e$ satisfies Assumption~\ref{phi consistency}. For the complement property, recognize that since $\partial A = \partial (A^\comp)$, $\Lambda_\e^0(A) = \Lambda_\e^1(A^\comp)$ and $\tilde\Lambda^0_\e(A) = \tilde\Lambda^1_\e(A^\comp)$ by definition. As for $\Lambda$-monotonicity, we can verify these four statements directly.
\begin{enumerate}
    \item[(i)] If $x\in\tilde\Lambda^0_\e(A)$, then $d(x,A\setminus E)\ge d(x,A)\ge \e$ so $x\in \tilde\Lambda^0_\e(A\setminus E)$.
    \item[(ii)] If $x\in \tilde\Lambda^1_\e(E)$, then $d(x,A\setminus E)\ge d(x,E^\comp)\ge \e$ so $x\in \tilde\Lambda^0_\e(A\setminus E)$.
    \item[(iii)] If $x\in \Lambda^0_\e(E) \cap A$, then $d(x,(A\setminus E)^\comp)\le d(x,E)<\e$ so $x\in \Lambda^1_\e(A\setminus E)$.
    \item[(iv)] If $x\in \Lambda^1_\e(A)\cap E^\comp$, then $d(x,(A\setminus E)^\comp)\le d(x,A^\comp) < \e$ so $x\in \Lambda^1_\e(A\setminus E)$.
\end{enumerate}

   Now, we consider $\phi_{\e,p}$. By definition,\[ \Lambda_{\e,p}^1(A^\comp)= \{x\in A^\comp : \prob(x'\in (A^\comp)^\comp: x'\sim\mathfrak p_{x,\e}) > p\} = \Lambda_{\e,p}^0(A).\] Similarly, one can show $\tilde\Lambda^0_{\e,p}(A) = \tilde\Lambda^1_{\e,p}(A^\comp)$. Hence the complement property holds for $\phi_{\e,p}$. Now we consider $\Lambda$-monotonicity. To simplify notation, we let $\prob(x;A) \ceq \prob(x'\in A : x'\sim \mathfrak p_{x,\e})$. Examining each of the $\Lambda$-monotonicity properties, we find the monotonicity with respect to set inclusion of the probability function 
    \begin{itemize} 
       \item[(i)] If $x\in \tilde\Lambda^0_{\e,p}(A)$, then $\prob(x; A\setminus E) \le \prob(x;A)\le p$ so $x\in \tilde\Lambda^0_{\e,p}(A\setminus E)$.
       \item[(ii)] If $x\in \tilde\Lambda^1_{\e,p}(E)$, then $\prob(x;A\setminus E) \le \prob(x; E^\comp) \le p$ so $x\in \tilde\Lambda^0_{\e,p}(A\setminus E)$.
        \item[(iii)] If $x\in \Lambda_{\e,p}^0(E)\cap A$,  then $\prob(x;(A\setminus E)^\comp) \ge \prob(x; E) >p$ so $x\in \Lambda^1_{\e,p}(A\setminus E)$.
        \item[(iv)] If $x\in \Lambda_{\e,p}^1(A)\cap E^\comp$, then $\prob(x; (A\setminus E)^\comp) \ge \prob(x;A^\comp) >p$ so $x\in \Lambda^1_{\e,p}(A\setminus E)$.
    \end{itemize} 
    Thus, $\phi_{\e,p}$ satisfies $\Lambda$-monotonicity and Assumption~\ref{phi consistency}.

    \label{rem:verify assumption}
\end{remark}

\begin{table}[!ht]
\centering
\begin{tabular}{|c|c|}
\hline 
If $x\in U_i$ & Then for $\Lambda(A \backslash E)$ we have  \\
\hline 
$U_1\ceq\tilde\Lambda^1_\phi(A) \cap \tilde\Lambda^0_\phi(E)$ & N/A  \\
$U_2\ceq\tilde\Lambda^1_\phi(A) \cap {\Lambda^0_{\phi}}(E)$ & $x\in \Lambda^1_\phi(A\setminus E)$ \\ 
$U_3\ceq\tilde\Lambda^1_\phi(A) \cap {\Lambda^1_{\phi}}(E)$ & N/A  \\
$U_4\ceq\tilde\Lambda^1_\phi(A) \cap \tilde\Lambda^1_\phi(E)$ &  $x\in \tilde\Lambda_\phi^0(A\setminus E)$  \\ 
$U_5\ceq\Lambda^1_{\phi}(A) \cap \tilde\Lambda^1_\phi(E)$ &  $x\in \tilde\Lambda_\phi^0(A\setminus E)$ \\ 
$U_6\ceq \Lambda^1_{\phi}(A) \cap \Lambda^1_{\phi}(E)$ & N/A  \\
$U_7\ceq \Lambda^1_{\phi}(A) \cap \Lambda^0_{\phi}(E)$ & $x\in \Lambda_\phi^1(A\setminus E)$  \\ 
$U_8\ceq\Lambda^1_{\phi}(A) \cap \tilde\Lambda^0_\phi(E)$ & $x\in \Lambda_\phi^1(A\setminus E)$\\ 
$U_9\ceq{\Lambda^0_{\phi}}(A) \cap {\Lambda^0_{\phi}}(E)$ & N/A \\
$U_{10}\ceq{\Lambda^0_{\phi}}(A) \cap \tilde\Lambda^0_\phi(E)$ & N/A \\
$U_{11}\ceq\Lambda^0_{\phi}(A) \cap {\Lambda^1_{\phi}}(E)$ & N/A  \\
$U_{12}\ceq\Lambda^0_{\phi}(A) \cap \tilde\Lambda^1_\phi(E)$ & $x\in \tilde\Lambda_\phi^0(A\setminus E)$ \\
$U_{13}\ceq \tilde\Lambda^0_\phi(A)$ & $x\in \tilde\Lambda_\phi^0(A\setminus E)$ \\
\hline
\end{tabular}
\caption{This table defines the 13 $U_i$ sets and exhibits all possible conclusions about the $\Lambda$-sets for $A\setminus E$ based on the $\Lambda$-sets for $A$ and $E$ from $\Lambda$-monotonicity. This set decomposition, along with the further refinement in \eqref{eqn:U hat and tilde}, will be key in proving the energy exchange inequality.}
\label{table:lambda}
\end{table}

\begin{remark}[$\Lambda$-set Decompositions]
\label{rem:lambda decomp}
    
Under Assumption~\ref{phi consistency}, we may decompose $\R^d$ in terms of the $\Lambda$-sets for $A,E\in \cB(\R^d)$ according to $\Lambda$-monotonicity. In doing so, we define the sets $U_1,\dots, U_{13}$, which partition $\R^d$ (see Table~\ref{table:lambda} and Figure~\ref{fig:ePerregions}). 

For the sets $U_i$ where no conclusion can be made about $\phi(x;A\setminus E)$, we will further decompose them into two subsets based on the $\phi$ values, i.e. \begin{equation}\widetilde U_i = \{x\in U_i : \phi(x;A\setminus E) = 0\}, \ \ \widehat U_i = \{x\in U_i: \phi(x;A\setminus E) = 1\},\label{eqn:U hat and tilde}\end{equation} for $i = 1,3,6,9,10,$ and $11$. 

The auxiliary symbols are meant to help the reader group the terms. Notice that the $\widetilde U_i$ sets contain points that cannot be perturbed by the adversary into the other class for the classifier $A\setminus E$ in accordance with all $\tilde \Lambda$ sets also containing points that are unable to be attacked by the adversary. On the other hand, the $\widehat U_i$ sets contain only points that can be perturbed into the opposite class. 

With this decomposition, we can express the $\Lambda$-sets for $A\setminus E$ using the $U$ sets as follows:
\begin{align*}
\Lambda_{\phi}^0(A\setminus E) &= \widehat U_3\cup \widehat U_6\cup \widehat U_9\cup \widehat U_{10} \cup \widehat U_{11}, \\
\Lambda_{\phi}^1(A\setminus E) &= \widehat U_{1}\cup U_2\cup U_7\cup U_8,\\
\tilde\Lambda^0_\phi(A\setminus E) &= \widetilde U_{3}\cup U_4\cup U_5\cup \widetilde U_{6}\cup \widetilde U_{9}\cup \widetilde U_{10}\cup \widetilde U_{11}\cup U_{12}\cup U_{13}, \\
\tilde\Lambda^1_\phi(A\setminus E) &= \widetilde U_{1}.
\end{align*} 

Depending on extra structure imposed by the choice of $\phi$, sometimes we can conclude certain sets are empty. For example when $\phi = \phi_\e$ (see Remark~\ref{rem: spec attack fns}), we have $\widehat U_{1} = \emptyset, \widetilde U_{3} = \emptyset,$  and $\widetilde U_{10} =\emptyset$. In the case where such sets are unambiguous in terms of the values of $\phi(x;A\setminus E)$, we drop the hat or tilde notation. However, $U_6,U_9$ and $U_{11}$ still requires a finer decomposition. Note that generally $\widetilde U_6, \widetilde U_9 =\emptyset$, but when boundaries of $A$ and $E$ intersect at more than discrete points then these sets can be non-empty. When $\widetilde U_6, \widetilde U_9 = \emptyset$ (such as in Figure~\ref{fig:ePerregions}), we also drop the tilde notation and let $U_6 = \widehat U_6$ and $U_9 = \widehat U_9$. The claims made here are verified in Appendix~\ref{app:phi eps claims}
\label{rem:U sets}
 \end{remark}

\begin{figure}[!ht]
    \centering
    \scalebox{.75}{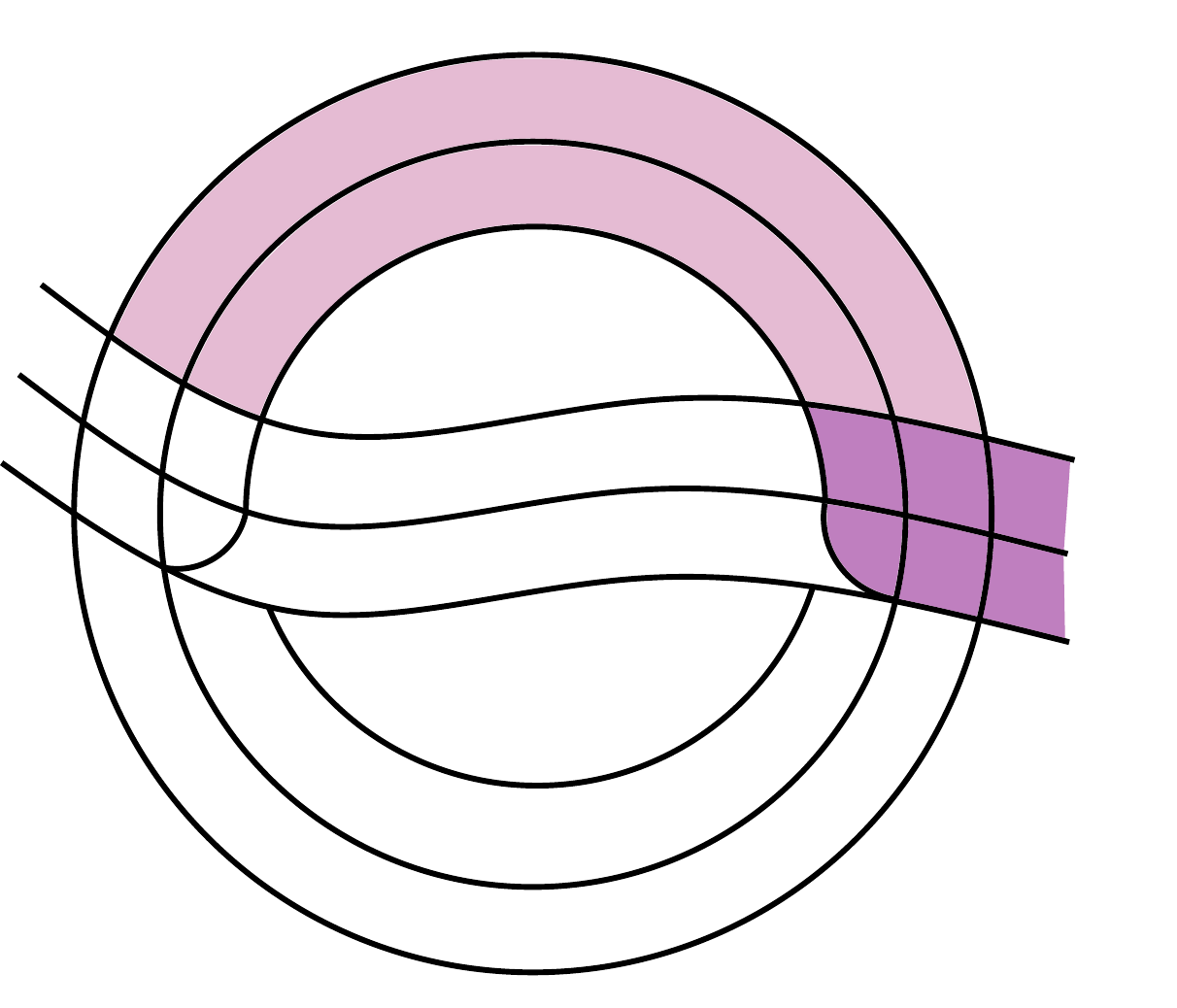}
    \caption{This diagram depicts the $U_i$ regions for the attack function $\phi_\eps$ associated with adversarial training problem \eqref{eqn:ATP}. The $\varepsilon$-perimeter regions of $A$ are shaded blue and purple whereas $\varepsilon$-perimeter regions of $A\setminus B_\dist(R)$ are shaded pink and purple. Note that some sets, such as $\widehat U_{1}$, are null sets for the $\e$-perimeter, and so do not appear in this figure.}
    \label{fig:ePerregions}
\end{figure}

Having stated our assumptions on $\phi$, we now turn to proving the first main result. In the following proposition, we examine the difference in energy between classifiers $A$ and $A\setminus E$ for $A,E\in \cB(\R^d)$ when $E$ belongs to a region where the label $0$ is energetically preferable according to the Bayes risk. We refer to the resulting inequality as the \textit{energy exchange inequality} because it quantifies the effect of removing the set $E$ from a classifier $A$ by examining the difference in risks. 

\begin{proposition}[Energy Exchange Inequality]
    Let $\phi$ be a deterministic attack function that satisfies Assumption~\ref{phi consistency}, let $A,E\in\cB(\R^d)$, and assume that $w_0\rho_0-w_1\rho_1 > \delta>0$ on $E$. If $J_\phi(A\setminus E) - J_\phi(A) \ge 0$, then
    \begin{equation} D_\phi(A;E) \le D_{\phi}(E^\comp;A) -\delta \mathcal L^d(A\cap E) + w_0\rho_0(\widehat U_{11}) + w_1\rho_1(\widehat U_{1}), \label{deficit inequality} \end{equation} where $U_{1}^1$ and $\widehat U_{11}$ are defined in Table~\ref{table:lambda}, namely $\widehat U_{1} = \{x\in \tilde\Lambda^1_\phi(A) \cap \tilde\Lambda^0_\phi(E): \phi(x;A\setminus E) = 1\}$ and $\widehat U_{11}= \{x\in {\Lambda_\phi^0}(A)\cap {\Lambda_\phi^1}(E):  \phi(x;A\setminus E) = 1\}$. 
    \label{generalInequality}
\end{proposition}

\begin{proof}
By \eqref{Jphi}, we have 
\begin{align*} J_\phi(A) &= w_0\rho_0(A\cup\Lambda_\phi^0(A)) + w_1\rho_1(A^\comp\cup \Lambda_\phi^1(A)),\\
J_\phi(A\setminus E) &=  w_0\rho_0((A\setminus E)\cup \Lambda_\phi^0(A\setminus E)) + w_1\rho_1((A\setminus E)^\comp\cup \Lambda_\phi^1(A\setminus E)).\end{align*}

Based on Remark~\ref{rem:U sets} with further details shown in Appendix~\ref{app:Lambda sets}, we can express $A\cap E$ and the sets comprising $J_\phi(A\setminus E)$ as
\begin{align}
A\cap E &= U_3\cup U_4\cup U_5 \cup U_6,\label{eqn:intersection}\\
\Lambda_{\phi}^0(A\setminus E) &= \widehat U_3\cup \widehat U_6\cup \widehat U_9\cup \widehat U_{10} \cup \widehat U_{11}, \\
\Lambda_{\phi}^1(A\setminus E) &= \widehat U_{1}\cup U_2\cup U_7\cup U_8,\\
A\setminus E &= U_1\cup U_2\cup U_7\cup U_8, \\
(A\setminus E)^\comp &= U_3\cup U_4\cup U_5\cup U_6\cup U_9\cup U_{10}\cup U_{11}\cup U_{12}\cup U_{13}.
\end{align}

We can write the adversarial deficit terms as
\begin{align}
    D_\phi(A;E)&= w_0\rho_0(U_{11}\cup U_{12}) + w_1\rho_1(U_5\cup U_6), \label{Dphi1} \\ 
    D_\phi(E^\comp;A)&= w_0\rho_0(U_{3}\cup U_{6}) + w_1\rho_1(U_2\cup U_7).
    \label{Dphi2}
\end{align}

Then we estimate, 
\begin{align*}
J_\phi(A\setminus E) -J_\phi(A) &=  w_0\rho_0(U_1\cup U_2\cup \widehat U_3 \cup \widehat U_6\cup U_7\cup U_8\cup \widehat U_9 \cup \widehat U_{10}\cup \widehat U_{11}) \\ 
&+ w_1\rho_1(\widehat U_{1}\cup U_2 \cup U_3 \cup U_4 \cup U_5\cup U_6\cup U_7\cup U_8\cup U_9\cup U_{10}\cup U_{11}\cup U_{12}\cup U_{13})\\
&-w_0\rho_0(U_1\cup U_2\cup U_3\cup U_4 \cup U_5\cup U_6\cup  U_7 \cup U_8\cup U_9 \cup U_{10}\cup U_{11}\cup U_{12}) \\ 
&- w_1\rho_1(U_5\cup U_6\cup U_7\cup U_8\cup U_9\cup U_{10}\cup U_{11}\cup U_{12}\cup U_{13})\\
&\le w_0\rho_0(\cancel{U_1}\cup \cancel{U_2}\cup \uuline{U_3} \cup \uuline{U_6}\cup \cancel{U_7}\cup \cancel{U_8}\cup {\cancel{\widehat U_9}} \cup \cancel{\widehat U_{10}}\cup \widehat U_{11}) \\ 
&+ w_1\rho_1(\widehat U_{1}\cup \uuline{U_2} \cup \uwave{U_3}\cup \uwave{U_4} \cup \uwave{U_5}\cup \uwave{U_6}\cup \uuline{U_7}\cup \cancel{U_8}\cup \cancel{U_9}\cup \cancel{U_{10}}\cup \cancel{U_{11}}\cup \cancel{U_{12}}\cup \cancel{U_{13}})\\
&-w_0\rho_0(\cancel{U_1}\cup \cancel{U_2}\cup \uwave{U_3}\cup \uwave{U_4} \cup \uwave{U_5}\cup \uwave{U_6}\cup  \cancel{U_7} \cup \cancel{U_8} \cup (\widetilde U_{9} \cup \cancel{\widehat U_9})  \cup (\widetilde U_{10} \cup \cancel{\widehat U_{10}}) \cup \uline{U_{11}}\cup \uline{U_{12}}) \\ 
&- w_1\rho_1( \uline{U_5}\cup \uline{U_6}\cup U_7\cup \cancel{U_8}\cup \cancel{U_9}\cup \cancel{U_{10}}\cup \cancel{U_{11}}\cup \cancel{U_{12}}\cup \cancel{U_{13}})\\
&\le \uuline{D_{\phi}(E^\comp;A)} - \uline{D_{\phi}(A;E)} - (w_0\rho_0 - w_1\rho_1)\uwave{(A\cap E)} + w_0\rho_0(\widehat U_{11}) + w_1\rho_1(\widehat U_{1}).
\end{align*}
In the last line, the inequality results from neglecting all remaining terms with a negative sign. As $J_\phi(A\setminus E) - J_\phi(A) \ge 0$ and $w_0\rho_0 - w_1\rho_1 > \delta > 0$ on $E$, we estimate
\begin{align*}
D_{\phi}(A;E) &\le D_{\phi}(E^\comp;A) - (w_0\rho_0 - w_1\rho_1){(A\cap E)} + w_0\rho_0(\widehat U_{11}) + w_1\rho_1(\widehat U_{1})\\
&< D_{\phi}(E^\comp;A) - \delta\mathcal L^d(A\cap E) + w_0\rho_0(\widehat U_{11}) + w_1\rho_1(\widehat U_{1}).
\end{align*}
\end{proof}

Observe that if $A\in\cB(\R^d)$ is a minimizer of $J_\phi$ for some deterministic attack function $\phi$, then $J_\phi(A\setminus E) - J_\phi(A) \ge 0$ for any $E\in \cB(\R^d)$ and Proposition~\ref{generalInequality} applies. This will be the setting for our results, although we state the result in its most general form here.

In later energy arguments, it will be helpful to express the difference in classification risks exactly instead of combining terms to form $D_\phi(A;E), D_\phi(E^\comp;A)$, and $\mathcal L^d(A\cap E)$. In Corollary~\ref{cor:exact diff}, we consider the same computation for $J_\phi(A\setminus E)-J_\phi(A)$ but now aim to simplify the difference as much as possible.


\begin{corollary}
Let $A\in\cB(\R^d)$ be a classifier for the generalized adversarial training problem and let $E \in \cB(\R^d)$.
Then using the same notation as in Proposition~\ref{generalInequality} and under the same assumptions,  \begin{align*}J_\phi(A\setminus E) - J_\phi(A) &= w_1\rho_1(\widehat U_{1}\cup U_2\cup \widehat U_3) - (w_0\rho_0-w_1\rho_1)(\widetilde U_{3}\cup U_4) \\
&- w_0\rho_0(U_5 \cup \widetilde U_{6} \cup \widetilde U_{9}\cup \widetilde U_{10} \cup \widetilde U_{11}\cup U_{12}).\end{align*}
    \label{cor:exact diff}
\end{corollary}

\begin{proof}
Let all sets $U_i, \widehat U_i, \widetilde U_i$ be as defined in Table~\ref{table:lambda} and \eqref{eqn:U hat and tilde}. We compute the exact difference in energies as follows:
\begin{align*} J_\phi(A\setminus E) -J_\phi(A) &= w_0\rho_0(\cancel{U_1}\cup \cancel{U_2}\cup \cancel{\widehat U_3} \cup \cancel{\widehat U_6}\cup \cancel{U_7}\cup \cancel{U_8}\cup \cancel{\widehat U_9} \cup \cancel{\widehat U_{10}}\cup \cancel{\widehat U_{11}}) \\ 
&+ w_1\rho_1(\widehat U_{1}\cup U_2 \cup U_3 \cup U_4 \cup \cancel{U_5}\cup \cancel{U_6}\cup \cancel{U_7}\cup \cancel{U_8}\cup \cancel{U_9}\cup \cancel{U_{10}}\cup \cancel{U_{11}}\cup \cancel{U_{12}}\cup \cancel{U_{13}})\\
&-w_0\rho_0(\cancel{U_1}\cup \cancel{U_2}\cup (\widetilde U_{3} \cup \cancel{\widehat U_3}) \cup U_4 \cup U_5\cup (\widetilde U_{6} \cup \cancel{\widehat U_6}) \cup  \cancel{U_7} \cup \cancel{U_8}\cup (\widetilde U_{9} \cup \cancel{\widehat U_9}) \ldots \\
&\ \ \ \ \ \ \ \ldots \cup (\widetilde U_{10} \cup \cancel{\widehat U_{10}}) \cup (\widetilde U_{11} \cup \cancel{\widehat U_{11}}) \cup U_{12}) \\ 
&- w_1\rho_1(\cancel{U_5}\cup \cancel{U_6}\cup \cancel{U_7}\cup \cancel{U_8}\cup \cancel{U_9}\cup \cancel{U_{10}}\cup \cancel{U_{11}}\cup \cancel{U_{12}}\cup \cancel{U_{13}})\\
&= w_1\rho_1(\widehat U_{1} \cup U_2 \cup \widetilde U_{3} \cup \widehat U_3 \cup U_4)\\
&- w_0\rho_0(\widetilde U_{3} \cup U_4\cup U_5 \cup \widetilde U_{6} \cup \widetilde U_{9}\cup \widetilde U_{10} \cup \widetilde U_{11}\cup U_{12})\\
&=w_1\rho_1(\widehat U_{1}\cup U_2\cup \widehat U_3) - (w_0\rho_0-w_1\rho_1)(\widetilde U_{3}\cup U_4) \\
&- w_0\rho_0(U_5 \cup \widetilde U_{6} \cup \widetilde U_{9}\cup \widetilde U_{10} \cup \widetilde U_{11}\cup U_{12}). 
\end{align*}
\end{proof}

In the following pair of corollaries, we will apply Proposition~\ref{generalInequality} to the adversarial training problem \eqref{eqn:ATP} and the probabilistic adversarial training problem \eqref{PATP}. 

\begin{corollary}
    Let $\varepsilon >0$ and $\phi = \phi_\varepsilon$. Let $A,E\in\cB(\R^d)$ such that $w_0\rho_0-w_1\rho_1 > \delta>0$ on $E$ and $J_{\varepsilon}(A\setminus E) - J_{\varepsilon}(A) \ge 0$. Then  \begin{equation} \varepsilon \Per_{\varepsilon}(A;E) \le \varepsilon \Per_{\varepsilon}(E^\comp;A) - \delta \mathcal L^d(A\cap E) + w_0\rho_0(\widehat U_{11}), \label{eps inequal} \end{equation}
    where $\widehat U_{11}= \{x\in A^\mathsf{c}\cap E: \dist(x, A\setminus E)<\varepsilon\}$.
    \label{eps energy inequality}
\end{corollary} 

\begin{proof} 
To prove the corollary, we only need to check that $\phi_\e$ satisfies Assumption~\ref{phi consistency} (which is done in Remark~\ref{rem:verify assumption}) and to verify that $\widehat U_{1}$ is empty. To that end, if $x \in \widehat U_{1}$ then
\[
x \in A \cap E^c \text{ such that } d(x,A^c) > \eps \text{ and } d(x,E) > \eps,
\]
which in turn implies that $d(x,A^\comp \cup E) > \eps$. Hence for such $x$, $\phi_\eps(x;A\setminus E) = 0$ and accordingly $\widehat U_{1} = \emptyset$.

\end{proof}

\begin{corollary}
    Let $\varepsilon>0$, $p\in[0,1)$, $\{\mathfrak p_{x,\varepsilon}\}_{x\in \R^d}$ be a family of probability measures, and $\phi = \phi_{\varepsilon,p}$. Let $A,E\in \cB(\R^d)$ such that $w_0\rho_0-w_1\rho_1 > \delta>0$ on $E$ and $J_{\varepsilon,p}(A\setminus E) - J_{\varepsilon,p}(A) \ge 0$. Then, \begin{equation} \ProbPer_{\varepsilon,p}(A;E) \le \ProbPer_{\e,p}(E^\comp;A) - \delta \mathcal L^d(A\cap E) + w_0\rho_0(\widehat U_{11}) + w_1\rho_1(\widehat U_{1}), \label{prob inequal}\end{equation}
    where $\widehat U_{11}= \{x\in A^\mathsf{c}\cap E:  \prob(x'\in A\setminus E:x'\sim\mathfrak p_{x,\e}) > p\}$ and $\widehat U_{1} = \{x\in \tilde\Lambda^1_{\e,p}(A) \cap \tilde\Lambda^0_{\e,p}(E): \prob(x'\in (A\setminus E)^\comp:x'\sim\mathfrak p_{x,\e})>p\}$. 
    
    \label{lem: prob energy inequality}
\end{corollary}

\begin{proof}
We verified that $\phi_{\e,p}$ satisfies Assumption~\ref{phi consistency} in Remark~\ref{rem:verify assumption}. Thus we can apply Proposition~\ref{generalInequality} to conclude that the energy exchange inequality holds for the probabilistic adversarial training problem \eqref{PATP}. 
    
\end{proof}


\section{Uniform Convergence for the Adversarial Training Problem}\label{sec:conv atp}

Before tackling convergence for the generalized adversarial problem \eqref{DATP}, we first consider the convergence for the adversarial training problem \eqref{eqn:ATP} to understand the results in a more concrete setting. The results for \eqref{eqn:ATP} are also stronger than those for \eqref{DATP} and allow for more straightforward proofs that provide the basis for our approach in the subsequent section. We will return to \eqref{DATP} in Section~\ref{sec:det conv} equipped with better intuition and understanding.

In this section, we establish uniform convergence in the Hausdorff metric of minimizers of the adversarial training problem \eqref{eqn:ATP} to Bayes classifiers on compact sets as the parameter $\varepsilon \to 0^+$. As previously stated in Remark~\ref{rem:atp related work}, current convergence results are in the (weaker) $L^1$ topology. We begin by stating a modest assumption we make upon the underlying metric space.

\begin{assumption}
 For the remainder of the paper, we assume that the metric $\dist$ is induced by a norm. Then, $\mathcal L^d(B_\dist(r)) \ceq \omega_\dist r^d$ for the constant $\omega_\dist = \mathcal L^d(B_\dist(1))$. Naturally, $\omega_\dist$ will also depend on the dimension $d$, but we suppress this in the notation. Additionally, we will identify the conditional measures in \eqref{eps perim} with their densities, meaning that we can express $d\rho_i = \rho_i(x) \, dx$. 
\end{assumption}


For these norm balls, it will be useful to estimate their $\e$-perimeter. When $\e\le R$ and  $\rho_0,\rho_1$ are bounded from above, this amounts to estimating the volume between two norm balls that are distance $2\e$ apart. 

\begin{lemma}
    Let $0 < \e \le R$ for some fixed $R>0$. Suppose $\rho_0,\rho_1 \le M$ on $\R^d$. Then, there exists a constant $\alpha >0$ independent of $R, \e$, and $x$ such that
    \[\e\Per_{\varepsilon}(B_\dist(x,R)) \le \alpha R^{d-1}\e. \]
    \label{lem: eps per ball est}
\end{lemma}
\begin{proof}
Recall that \eqref{eps perim} for $A = B_\dist(x,R)$ gives \[\e\Per_{\varepsilon}(B_\dist(x,R)) = w_0\rho_0(B_\dist(x,R+\e)\setminus B_\dist(x,R))+ w_1\rho_1(B_\dist(x,R)\setminus B_\dist(x,R-\e)).\] As $\rho_0,\rho_1$ are bounded from above by $M$, \[\e\Per_{\varepsilon}(B_\dist(x,R)) \le M\mathcal L^d(B_\dist(x,R+\e) \setminus B_\dist(x,R-\e)) = M(\mathcal L^d(B_\dist(x,R+\e)) - \mathcal L^d(B_\dist(x,R-\e))).\] 

By the scaling properties of the norm ball, $\mathcal L^d(B_\dist(x,r)) = \omega_\dist(r)^d$ for all $r\ge 0$. By convexity, we estimate
\[
(R+\e)^d -(R-\e)^d \le d(R+\e)^{d-1}2\e.
\]
As $\e\le R$, we conclude \[\e\Per_{\varepsilon}(B_\dist(x,R)) \le M\omega_\dist d(2R)^{d-1}2\e\le \alpha R^{d-1}\e.\] 

\end{proof}

Throughout the paper, we will require an upper bound on the $\e$-perimeter of the \textit{complement} of $B_\dist(x,R)$. By the complement property from Assumption~\ref{phi consistency} (verified to hold for the $\e$-perimeter in Remark~\ref{rem:verify assumption}),  the bound given by Lemma~\ref{lem: eps per ball est} still holds for $\e\Per_{\varepsilon}(B_\dist(x,R)^\comp)$ since the same upper bound is true for $\rho_0$ and $\rho_1$, namely, 
\begin{equation} \e\Per_{\varepsilon}(B_\dist(x,R)^\comp) \le \alpha R^{d-1}\e.     
    \label{ineq:ball eper}
    \end{equation}

With our normed setting clear, we begin the process of proving uniform Hausdorff convergence for minimizers of the adversarial training problem \eqref{eqn:ATP}. The first step involves proving a technical lemma about the interaction between minimizers and sets $B_\dist(x,R) \subset \{ w_0\rho_0 - w_1 \rho_1 > \delta > 0\}$. Importantly, this means $B_\dist(x,R) \cap A_0=\emptyset$ for a Bayes classifier $A_0$, which can help us relate minimizers of the adversarial training problem to Bayes classifiers. By applying a slicing argument, we will show that minimizers are disjoint from $B_{\dist}(x,R/2^{d+1})$.  

\begin{lemma}
    Let $A\in\cB(\R^d)$ be a minimizer of the adversarial training problem \eqref{eqn:ATP} for $\varepsilon > 0$. Suppose there exists $x\in \R^d$ and $R>0$ such that $w_0\rho_0 - w_1\rho_1 > \delta > 0$ on $B_{\dist}(x,2R)$ with $\rho_0,\rho_1 \le M$  on $\R^d$. Then, there exists a $C > 0$ independent of $R,\delta,\e$, and $x$ such that if $\varepsilon \le \min\left\{R/2^{d+2}, CR\delta^{d+1}\right\}$, then $A \cap B_{\dist}(x,R/2^{d+1}) = \emptyset$.
    
\label{BR slices}
\end{lemma}

\begin{proof}
Fix $\e>0$. Choose a coordinate system such that $x =0$ and write $B_\dist(0,R) = B_\dist(R)$. For the sake of contradiction, suppose there exists $z\in A\cap B_\dist(R/2^{d+1})$. Then, \begin{equation}
    \mathcal L^d(A^\varepsilon \cap B_\dist(R/2^d)) \ge \mathcal L^d(B_\dist(\varepsilon)) = \omega_\dist \varepsilon^d.
    \label{slices cond}\end{equation}

Corollary~\ref{eps energy inequality} shows for $r \le R$,
\[\e \Per_\e(A; B_\dist(r)) \le \e \Per_\e(B_\dist(r)^\comp; A) - \delta \mathcal L^d(A \cap B_\dist(r)) + w_0\rho_0(\widehat U_{11})\] with $\widehat U_{11} \subset \Lambda_\e(B_\dist(r)) \cap A^\e$ and $w_0\rho_0(\widehat U_{11}) \le \e \Per_\e( B_\dist(r)^\comp;A^\e)$.

In particular, using the fact that $w_0\rho_0 > \delta$ in $B_\dist(R)$, we obtain
\begin{align*}
\mathcal L^d&((A^\e \setminus A) \cap B_\dist(R)) \le \frac{w_0}{\delta}\rho_0((A^\e \setminus A) \cap B_\dist(R))\\
&\le \frac{\e}{\delta}\Per_\e(A;B_\dist(R))\le \frac{\e}{\delta} \Per_\e(B_\dist(R)^\comp;A)-\mathcal L^d(A\cap B_\dist(R))+ \frac{w_0}{\delta}\rho_0(\widehat U_{11}). 
\end{align*}
Rearranging and applying the bound $w_0\rho_0(\widehat U_{11}) \le \e \Per_\e(B_\dist(R)^\comp;A^\e)$, we estimate
\begin{align}
\mathcal L^d(A^\e \cap B_\dist(R))&\le \frac{\e}{\delta}\Per_\e(B_\dist(R)^\comp;A)+ \frac{w_0}{\delta}\rho_0(\widehat U_{11})\\
&\le \frac{2\e}{\delta}\Per_\e(B_\dist(R)^\comp;A^\e)\\
&\le2\alpha \frac{R^{d-1}}{\delta}\e  \label{est for R}
\end{align}
with the last inequality due to \eqref{ineq:ball eper}. Note that \eqref{est for R} also holds replacing $R$ with $r \le R.$

Using that $\rho_0,\rho_1$ are bounded from above by $M$, we estimate
\begin{align*}
\sum_{k=0}^{\lfloor\frac{R}{4\e}\rfloor-1} &\mathcal L^d(A^\e \cap B_\dist(R/2 + 2k\e)) \le\sum_{k=0}^{\lfloor\frac{R}{4\e}\rfloor-1} \frac{2\e}{\delta}\Per_\e(B_\dist(R/2 + 2k\e)^\comp; A^\e)\\
&\le \frac{2M}{\delta} \sum_{k=0}^{\lfloor\frac{R}{4\e}\rfloor-1} \mathcal L^d(A^\e \cap(B_\dist(R/2+(2k+1)\e)\setminus B_\dist(R/2+(2k-1)\e))\\
&\le \frac{2M}{\delta} \mathcal L^d(A^\e \cap B_\dist(R)) \le 4\alpha M\frac{R^{d-1}}{\delta^2}\e
\end{align*}
thanks to \eqref{est for R}.

In particular,
\[ \left\lfloor\frac{R}{4\e}\right\rfloor \min_k \mathcal L^d(A^\e \cap B_\dist(R/2+2k\e))\le \frac{2M}{\delta} \mathcal L^d(A^\e\cap B_\dist(R))\le 4\alpha M\frac{R^{d-1}}{\delta^2}\e.\]

If $\e \le R/8$ so that $\lfloor\frac{R}{4\e}\rfloor\ge \frac{R}{4\e}-1\ge \frac{R}{8\e}$, then by letting $s_1 =R/2+2k\e$ achieve $\min_k \mathcal L^d(A^\e \cap B_\dist(R/2+2k\e))$, we then obtain 
\[\mathcal L^d (A^\e \cap B_\dist(s_1))\le 32\alpha M\frac{R^{d-2}}{\delta^2}\e^2. \label{est for s1}\]

Then, repeating the same construction at the scale $R/2^i$, $i \ge 2$, we find
\[\mathcal L^d(A^\e \cap B_\dist(s_i))\le \frac{8\e}{R/2^{i-1}} \frac{2M}{\delta}\mathcal L^d(A^\e \cap B_\dist(s_{i-1}))\le \frac{2^{i+3}M}{R\delta\e} \mathcal L^d(A^\e \cap B_\dist(s_{i-1}))\]
as long as $\e \le \frac{R}{2^{i+2}}$ (that is, $i \le \log_2 \left(\frac{R}{4\e}\right)$).

For $i = d$, it follows
\begin{align*}
\mathcal L^d(A^\e \cap B_\dist(s_d)) &\le 2^{\sum_{i=2}^d i}\left(\frac{8M\e}{R\delta}\right)^{d-1}\mathcal L^d(A^\e \cap B_\dist(s_1))\\
&\le 2^{\frac{d(d+1)}{2}+3d-4}\left(\frac{M\e}{R\delta}\right)^{d-1}32 \alpha M\frac{R^{d-2}}{\delta^2}\e^2.
\end{align*}

Hence,
\[\mathcal L^d(A^\e \cap B_\dist(s_d)) \le 2^{\frac{d(d+1)}{2}+3d+1}\alpha M^d\frac{\e^{d+1}}{R\delta^{d+1}}.\]

Letting $C_{d+1}\ceq 2^{\frac{d(d+1)}{2}+3d+1}\alpha M^d$, we conclude if $\e < \min\{R/2^{d+2}, \omega_dC^{-1}_{d+1}R\delta^{d+1}\},$ then \[\mathcal L^d(A^\e \cap B_\dist(R/2^d))\le \mathcal L^d(A^\e \cap B_\dist(s_d))\le \frac{C_{d+1}}{R\delta^{d+1}}\e^{d+1}< \omega_d\e^d\] which implies that $A\cap B_\dist(R/2^{d+1}) = \emptyset$ by \eqref{slices cond}.

\end{proof}

\begin{remark}
    In Lemma~\ref{BR slices}, we can slightly relax the assumption that $A$ is a minimizer as follows: Recall that we assume $w_0\rho_0 - w_1\rho_1 > \delta > 0$ on $B_\dist(x,2R)$. If we have that $J_\varepsilon(A\setminus B_\dist(x,r)) - J_\varepsilon(A) \ge 0$ for all $r$ such that $R/2^{d+2} \le r \le R$, then the energy exchange inequality \eqref{eps inequal} still holds and the same proof for Lemma~\ref{BR slices} shows that $A\cap B_\dist(x,R/2^{d+1}) = \emptyset$. 
    \label{rem:relaxed assumptions}
\end{remark} 

We now aim to directly relate minimizers of the adversarial training problem \eqref{eqn:ATP} to Bayes classifiers. Recall that the maximal and minimal Bayes classifiers \eqref{minmax BC} are given by \[A_0^{\max} = \{x\in \R^d : w_0\rho_0(x) \leq w_1\rho_1(x)\},\qquad
A_0^{\min} = \{x\in \R^d : w_0\rho_0(x) < w_1\rho_1(x)\}.\] We will \textit{not} be assuming that $A_0^{\max}$ and $A_0^{\min}$ coincide up to a set of  $\rho$ measure zero unless explicitly stated. 

We will now show that on a compact set, we can ``corral'' the minimizer of the adversarial training problem \eqref{eqn:ATP} by any distance $\eta >0$, in the sense that it must lie between the $\eta$-dilation of $A_0^{\max}$ and the $\eta$-erosion of $A_0^{\min}$ when $\varepsilon$ is small enough. 

\begin{lemma}
 Let $A_0^{\max}$ be the maximal Bayes classifier. Suppose that $\rho_0,\rho_1$ are continuous and bounded from above on $\R^d$, and let $\eta > 0$. Then for any compact set $K\subset \R^d$, there exists an $\varepsilon_0>0$ such that for all $0<\varepsilon< \varepsilon_0$, \[\Big[A_\varepsilon \cap K \Big] \subset \Big[(A_0^{\max})^\eta \cap K\Big]\] where $A_\varepsilon\subset\R^d$ is an arbitrary minimizer of the adversarial training problem. 
\label{l-inf conv}
\end{lemma}

\begin{proof}
For convenience, we abuse notation and let $A_0 = A_0^{\max}$. Assume that $\left(A_0^\eta\right)^\comp \cap K \neq \emptyset$ as otherwise the result is trivial. The conditions are also trivially satisfied if $w_0\rho_0-w_1\rho_1$ never changes sign. This is because for all $\varepsilon> 0$ either $A_0 = A_\varepsilon = \emptyset$ if $w_0\rho_0 - w_1\rho_1 > 0$ on $\R^d$ or $A_0 = A_\varepsilon = \R^d$ otherwise.

Fix $\eta > 0$. Let $R = \frac{\eta}{3}$. Observe that $\overline{A_0^R}\cap \overline{K^{2R}}$ is compact and $A_0\subset \overline{A_0^R}$. Then, by the continuity of $w_0\rho_0-w_1\rho_1$ on $\overline{A_0^R}\cap \overline{K^{2R}}$, there exists a $\delta > 0$ such that \[ \Big[E_{\delta} \cap \overline{K^{2R}}\Big] \subset \Big[\overline{A_0^R}\cap \overline{K^{2R}}\Big]\] where $E_{\delta} = \{x\in \R^d: w_0\rho_0(x) - w_1\rho_1(x) \le \delta\}$. This implies $\ds\left[E_{\delta}^\comp \cap \overline{K^{2R}}\right] \supset \left[\left(\overline{A_0^R}\right)^\comp \cap \overline{K^{2R}}\right]$,  so $w_0\rho_0 - w_1\rho_1 > \delta > 0$ on $\left(\overline{A_0^R}\right)^\comp \cap \overline{K^{2R}}$. In particular as $(A_0^\eta)^\comp \cap K \subset \left[\left(\overline{A_0^R}\right)^\comp \cap \overline{K^{2R}}\right]$, the difference in densities $w_0\rho_0 - w_1\rho_1 > \delta $ on  $(A_0^\eta)^\comp \cap K$.

Take $x\in (A_0^{\eta})^\mathsf{c} \cap K$. Observe that $B_\dist(x,2R)$ satisfies the conditions of Lemma~\ref{BR slices} for $\delta$ as determined previously. Take $\varepsilon_0 = \min\left\{R/2^{d+2}, CR\delta^{d+1}\right\}$ for $C$ is independent of $R,\delta, \e,$ and $x$. Let $\varepsilon \le \varepsilon_0$ and let $A_\varepsilon$ be a minimizer of the adversarial training problem \eqref{eqn:ATP}. Then, $A_\varepsilon \cap B(x,R/2^{d+1}) = \emptyset$ for all $x\in (A_0^{\eta})^\mathsf{c} \cap K$ which implies that $A_\varepsilon \cap (A_0^{\eta})^\mathsf{c} \cap K = \emptyset$. Thus, we conclude \[\Big[A_\varepsilon \cap K\Big] \subset\Big[A_0^\eta\cap K\Big].\]

\end{proof}

\begin{remark}
    The only place where we use the compactness assumption in Lemma~\ref{l-inf conv} is to determine $\delta$ from $\eta$ by the continuity of $w_0\rho_0-w_1\rho_1$ a compact set. 
    \label{rem:compactness}
\end{remark}

The proof established that minimizers $A_\e$ of the adversarial training problem \eqref{eqn:ATP} can be corralled by the maximal Bayes classifier. We can also corral $A_\e$ by the minimal Bayes classifier as follows: Consider interchanging the densities so data points $x$ are distributed according to $\widetilde \rho_0 = \rho_1$ and  $\widetilde \rho_1 = \rho_0$. We can apply Lemma~\ref{l-inf conv} to the minimizer $\widetilde A_\varepsilon = (A_\eps)^\comp$ of the interchanged problem. We can conclude that for all compact sets $K\subset \R^d$ and $\eta > 0$ there exists an $\varepsilon_0 >0$, such that \[ \Big[(A_0^{\min})^{-\eta}\cap K\Big]\subset \Big[A_\varepsilon \cap K\Big] \] for all $\varepsilon\le\varepsilon_0$. This means that we have a two-sided, or ``corralling,'' bound on our minimizer for $\e$ small enough, namely  \[ \Big[(A_0^{\min})^{-\eta}\cap K\Big]  \subset \Big[A_\varepsilon \cap K\Big] \subset \Big[(A_0^{\max})^{\eta} \cap K\Big]. \]

The corralling argument will allow us to examine the Hausdorff distance between Bayes classifiers and minimizers of the adversarial training problem \eqref{eqn:ATP} as the adversarial budget decreases to zero. To begin, we recall the  definition of the Hausdorff distance.  

\begin{definition}
    The Hausdorff distance between two sets $A,E\subset \R^d$ is given by \[\dist_H(A,E) \coloneqq \max\left\{\sup_{x\in A}\dist(x, E), \sup_{x\in E}\dist(x, A)\right\}\] for a metric $\dist$ on $\R^d$. 
    Furthermore, $\dist_H$ is a pseudometric on $\cB(\R^d)$. 
\end{definition}

\begin{remark}
   If $\dist_H(A_0^{\max}, A_0^{\min}) = 0$, 
   then for any $\eta > 0$ and compact set $K\subset \R^d$, there exists an $\varepsilon_0 > 0$ such that 
    \[\Big[(A_0)^{-\eta}\cap K\Big]\subset \Big[A_\varepsilon \cap K\Big] \subset\Big[(A_0)^{\eta}\cap K\Big]\] for all $\varepsilon \le \varepsilon_0$ and for $A_0$ the unique Bayes classifier.
    \label{rem: containment unique BC}
\end{remark}

We now have the tools to show the uniform convergence of minimizers $A_\e$ of the adversarial training problem \eqref{eqn:ATP} to the Bayes classifier $A_0$. To begin, we prove the more general version of the result when the Bayes classifier is not unique up to a set of $\rho$ measure zero. In this case, we can only show that $\lim_{\e\to 0^+} A_\e$ must be corralled by the maximal and minimal Bayes classifiers.

\begin{theorem}
   Suppose $\rho_0, \rho_1$ are continuous and bounded from above on $\R^d$. Let $\{A_\varepsilon\}_{\varepsilon>0}$ be a sequence of minimizers of the adversarial training problem \eqref{eqn:ATP} for $\varepsilon\to 0^+$. Then, for any compact set $K\subset \R^d$, \[\lim_{\e\to0^+}\dist_H((A_{\varepsilon}\cup A_0^{\max})\cap K,  A_0^{\max}\cap K) = 0\
   \text{ and } \ \lim_{\e\to0^+} \dist_H((A_{\varepsilon}\cap A_0^{\min})\cap K,  A_0^{\min}\cap K)= 0.\]
    \label{hausdorff conv}
\end{theorem} 

\begin{proof}
   Let $K$ be a compact set. Observe that $A_0^{\max}\subset A_\varepsilon\cup A_0^{\max}$, so \[\dist_H((A_\varepsilon\cup A_0^{\max})\cap K, A_0^{\max}\cap K) = \sup_{x\in (A_\varepsilon\cup A_0^{\max})\cap K} \dist(x,A_0^{\max}\cap K).\] For the sake of contradiction, suppose this quantity does not go to zero as $\varepsilon\to0^+$. Then, there exists an $\eta >0$ such that for all $\varepsilon_0>0$ there exists an $0<\varepsilon\leq\varepsilon_0$ such that, \[\sup_{x\in (A_\varepsilon\cup A_0^{\max})\cap K}\dist(x, A_0^{\max}\cap K) >\eta.\] However, this contradicts Lemma~\ref{l-inf conv}. Thus, we conclude \[\lim_{\e\to0^+}\dist_H((A_{\varepsilon}\cup A_0^{\max})\cap K,  A_0^{\max}\cap K) = 0.\] As $A_{\varepsilon}\cap A_0^{\min} \subset A_0^{\min},$ an analogous argument proves that \[\lim_{\e\to0^+} \dist_H((A_{\varepsilon}\cap A_0^{\min})\cap K,  A_0^{\min}\cap K)= 0.\]

\end{proof}

\begin{corollary}
     Suppose that $\dist_H(A_0^{\max},A_0^{\min}) = 0$. Then under the same assumptions as Theorem~\ref{hausdorff conv},
    \[\lim_{\e\to0^+}\dist_H(A_\varepsilon \cap K, A_0\cap K) = 0\] for $A_0$ the unique Bayes classifier.
    \label{BC uniqueness}
\end{corollary}

\begin{proof}
    This follows from Theorem~\ref{hausdorff conv} as the result of Lemma~\ref{l-inf conv} simplifies when $\dist_H(A_0^{\max}, A_0^{\min}) =0$ as described in Remark~\ref{rem: containment unique BC}.
\end{proof}

In the case where  $\dist_H(A_0^{\max},A_0^{\min}) = 0$, it is natural to consider rates of convergence. In order to obtain such rates, we introduce the following assumption:

\begin{assumption}
    The level set $\{w_0\rho_0 = w_1\rho_1\}$ is \textit{non-degenerate}, meaning that $w_0\rho_1-w_1\rho_1 \in C^1(\R^d)$ and $|w_0\grad \rho_0 - w_1\grad \rho_1| > \alpha > 0$ on $\{w_0\rho_0 = w_1\rho_1\}$ for some constant $\alpha$. In this case, Bayes classifiers are unique up to a set of $\mathcal L^d$ measure zero and $\dist_H(A_0^{\max},A_0^{\min}) = 0$.
    \label{nondegen assumption}
\end{assumption} 

Now, we establish the convergence rate for minimizers of the adversarial training problem \eqref{eqn:ATP} to Bayes classifiers under this non-degeneracy assumption.

\begin{corollary}
    Suppose Assumption~\ref{nondegen assumption} holds and that $\rho_0,\rho_1$ are continuous and bounded from above on $\R^d$. For any compact set $K\subset \R^d$, there exists a constant $C>0$ such that \[ \limsup_{\varepsilon \to 0^+}\frac{\dist_H (A_\varepsilon \cap K, A_0\cap K)}{\varepsilon^{\frac{1}{d+2}}} \leq C \] where $A_0$ is the Bayes classifier.

    \label{coro:eps conv rate}
\end{corollary}

\begin{proof}
    Consider a sequence $\{\eta_i\}_{i\in \N}$ where $\eta_i > 0$ and $\eta_i \to 0^+$. Define $\varepsilon_i = \min\{C\eta_i, C\eta_i \delta_i^{d+1}\}$ based on the requirements on $\varepsilon$ from Lemma~\ref{BR slices} with $R = \eta_i$ and the continuity bound $\delta =\delta_i$ from Lemma~\ref{l-inf conv}. In this proof, $C$ is a constant always independent of $\eta_i, \varepsilon_i,$ and $\delta_i$ that we will allow to vary throughout this proof. 
    
    As $w_0\rho_0-w_1\rho_1 \in C^1(\R^d)$ and its gradient is bounded away from 0, the boundary $\partial A_0 = \{w_0\rho_0 = w_1\rho_1\}$ is a $C^1$ surface by the implicit function theorem, and hence the Hausdorff distance between the minimal and maximal sets is zero. Furthermore for $\eta_i \ll 1$, $\delta_i$ is the same order as $\eta_i$ which implies $\varepsilon_i= C\eta_i^{d+2}$. 
    
    For each $\varepsilon_i$, let $A_{\varepsilon_i}$ be the associated minimizer of the adversarial training problem \eqref{eqn:ATP}. By Theorem~\ref{hausdorff conv} along with Remark~\ref{BC uniqueness}, for any compact set $K\subset \R^d$ we have that 
    \[\dist_H(A_{\varepsilon_i}\cap K, A_0\cap K) < \eta_i = C\varepsilon_i^\frac{1}{d+2}.\] 

    Thus, we conclude that \[ \limsup_{\varepsilon_i \to 0^+}\frac{\dist_H (A_{\varepsilon_i} \cap K, A_0\cap K)}{\varepsilon_i^{\frac{1}{d+2}}} \leq C. \]

\end{proof}

\begin{remark}
Although we have shown the convergence rate to be at most $O(\varepsilon^{\frac{1}{d+2}})$, we expect that the convergence rate is actually $O(\varepsilon)$ (see the formal asymptotics near $\e = 0$ derived by \textcite{ngt2022NecCond}). The reason we get the convergence rate $\varepsilon^{\frac{1}{d+2}}$ is from the $\delta^{d+1}$ that appears in our bounds for $\varepsilon$. In Lemma~\ref{BR slices}, this term comes from the iterative argument that often employs crude volume bounds. More precise estimates would be required to improve the convergence rate.
\label{rem: conv rate}
\end{remark}

\section{Uniform Convergence for Other Deterministic Attacks}\label{sec:det conv}

Now, we will turn our focus to the generalized adversarial training problem \eqref{DATP}. At the end, we will present the results for the probabilistic adversarial training problem \eqref{PATP} as an example of our results for \eqref{DATP}. Unlike the case of the adversarial training problem \eqref{eqn:ATP}, existence of minimizers to \eqref{DATP} is an open question, and in this case our convergence result can be understood in the spirit of `a priori' estimates in partial differential equations. First, we will make precise which deterministic attack functions we consider. 

\begin{definition}
    A deterministic attack function $\phi$ is \textit{metric} if an adversary's attack on $x$ only depends upon points within distance $\varepsilon$ of $x$ for some adversarial budget $\e >0$.  More precisely for two classifiers $A,\widetilde A\in \cB(\R^d)$, 
    \[A\cap B_\dist(x,\e) = \widetilde A\cap B_\dist(x,\e) \implies \phi(x;A) = \phi(x;\widetilde A).\] To avoid a trivial situation where $x$ is always attacked independent of the choice of $A$, we assume the adversary has no power, meaning $\phi(x;A) \equiv 0$, if $A = \emptyset$ or $A =\R^d$ when $\phi$ is a metric attack function. 
    \label{def:metric attack}
\end{definition}

In the following pair of lemmas, we will show two important properties of metric attack functions. The first will allows us to relate $D_\phi$ with $\e\Per_\e$ and provides an upper bound on $D_\phi$ by $\e\Per_\e$. This will allow us to employ many of the estimates of $\e\Per_\e$ from Lemma~\ref{BR slices} in Lemma~\ref{general slices}.

\begin{lemma}
Let $\phi$ be a metric deterministic attack function. For any set $A, E\in \cB(\R^d)$ we have that \[D_\phi(A) \le \varepsilon\text{Per}_\varepsilon(A) \ \text{ and } \ D_\phi(A;E) \le \varepsilon\text{Per}_\varepsilon(A;E).\]
\label{deficit and eps per}
\end{lemma}

\begin{proof}
    It will be sufficient to show that $\Lambda_\phi^i(A)\subset \Lambda_\e^i(A)$. Take $x\in \Lambda_\phi^0(A)$. If we consider $\widetilde A = \emptyset$, the metric property states 
    \[A\cap B_\dist(x,\e) = \emptyset \implies \phi(x;A) = \phi(x;\emptyset) =0.\] For $x\in A^\comp,$ $A\cap B_\dist(x,\e) \neq \emptyset$ implies $x\in A^\e\setminus A = \Lambda_\e^0(A)$. Thus, we conclude $\Lambda_\phi^0(A)\subset \Lambda_\e^0(A)$. A similar argument with $\widetilde A = \R^d$ shows that $\Lambda_\phi^1(A)\subset \Lambda_\e^1(A)$. 
    
\end{proof}

We now prove a second property of metric attack functions, which isolates where the values of $\phi(x;A)$ and $\phi(x;A\setminus E)$ may differ.

\begin{lemma}
    Let $\phi$ be a metric deterministic attack function. For sets $A,E\in \cB(\R^d)$, if $x\in (E^\e)^\comp$ then $\phi(x;A) =\phi(x;A\setminus E)$. 
    \label{lem:phi isolation}
\end{lemma}

\begin{proof}
Suppose $x \in (E^\eps)^\comp$. Then, $B(x,\eps) \subset E^\comp$ and so $A \cap B(x,\eps) = (A \setminus E) \cap B(x,\eps)$. Hence, the metric property then implies that $\phi(x;A) = \phi(x;A\setminus E)$.

\end{proof}

We require one additional assumption on a metric attack function $\phi$ in order to prove the generalized version of Lemma~\ref{BR slices}. Namely, if the size of the intersection of $B_\dist(x,\e)$ with the opposite class of $x$ satisfies a lower bound, then $x\in \Lambda_\phi(A)$. 

\begin{assumption}
    Let $\phi$ be a metric deterministic attack function with budget $\eps>0$. For a classifier $A\in \cB(\R^d)$, we assume:
    \[ x\in A^\comp \text{ and }  \mathcal L^d(A\cap B_\dist(x,\e)) > \beta\varepsilon^d\implies  x\in \Lambda_{\phi}^0(A) , \] 
    \[  x\in A \text{ and } \mathcal L^d(A^\comp\cap B_\dist(x,\e)) > \beta\varepsilon^d  \implies x\in \Lambda_{\phi}^1(A)  ,\]
    for some constant $0<\beta<\omega_\dist$ independent of $x, \varepsilon,$ and $A$. 
    \label{phi assumptions}
\end{assumption}

As a consequence of this assumption, we have if $x\in \tilde\Lambda_\phi^0(A)$, then $\mathcal L^d(A\cap B_\dist(x,\e)) \le \beta\e^d$. Likewise, if $x\in \tilde\Lambda_\phi^1(A)$, then $\mathcal L^d(A^\comp \cap B_\dist(x,\e)) \le \beta\e^d$. Furthermore, if Assumption~\ref{phi consistency} also holds for $\phi$, then only one of the two lower bounds needs to be assumed as the other follows by the complement property.

\begin{remark}
This assumption states that a point $x\in \R^d$ is attacked if the portion of its $\e$-neighbors with the opposite label is on the order of $\varepsilon^d$. In this way, the deterministic attack function depends on the adversarial budget $\e$ and the metric. 

Observe that the adversarial training problem~\eqref{eqn:ATP} satisfies Assumption~\ref{phi assumptions}. In fact, it satisfies the statements \begin{align*}
x\in \Lambda_\e^0(A) &\iff x\in A^\comp \text{ and } A\cap B_\dist(x,\e) \neq \emptyset, \\
x\in \Lambda_\e^1(A) &\iff x\in A \text{ and } A^\comp\cap B_\dist(x,\e) \neq \emptyset.
\end{align*}
In Proposition~\ref{verify assump prob}, we will verify that the probabilistic adversarial training problem \eqref{PATP} also satisfies Assumption~\ref{phi assumptions}.

\label{rem:phi and eps}
\end{remark}

In order to show uniform convergence for the generalized adversarial training problem \eqref{DATP}, we first prove the analog of Lemma~\ref{BR slices} by a similar slicing argument. We leverage the relationship between the adversarial deficit and the $\e$-perimeter established in Lemma~\ref{deficit and eps per}. However, there are a few key differences in both the results and the proof. Whereas in Lemma~\ref{BR slices} we show that minimizers of the adversarial training problem \eqref{eqn:ATP} are disjoint from certain norm balls that are misclassified, we show that the intersection of minimizers of \eqref{DATP} with a misclassified norm ball must have $\mathcal L^d$ measure zero. In this sense, we establish a necessary condition for minimizers of \eqref{DATP}. As for the proof of the statement, the final step differs significantly between Lemma~\ref{BR slices} and Lemma~\ref{general slices}. In final step of Lemma~\ref{BR slices}, we are able to use the fact that a single point causes misclassification on the order of $\varepsilon^d$. For the general case, the lower bound on the $\mathcal L^d$ measure condition for misclassification from Assumption~\ref{phi assumptions} requires a more delicate energy argument that examines the exact difference in energies.

\begin{lemma} 
Let $\phi$ be a metric deterministic attack function for $\e > 0$ satisfying  Assumptions~\ref{phi consistency} and \ref{phi assumptions}. Suppose $\rho_0,\rho_1$ are continuous and bounded from above on $\R^d$. Furthermore, suppose $A\in \cB(\R^d)$ is a minimizer of the generalized training problem \eqref{DATP} and there exists $x\in \R^d$ and $R>0$ such that $w_0\rho_0-w_1\rho_1 > \delta > 0$ on $B_\dist(x,2R)$. Then, there exists a constant $C > 0$ independent of $R, \delta,\e$, and $x$ such that if $\varepsilon \le  \min\left\{R/2^{d+2}, CR\delta^{d+1}\right\}$, then $\mathcal L^d(A \cap B(x,R/2^{d+2})) = 0$. 
    \label{general slices}
\end{lemma}

\begin{proof}
    Choose a coordinate system such that $x = 0$ and write $B_\dist(0,R) = B_\dist(R)$ with $x$ as in the statement above. 

 We will first find an initial estimate for $\mathcal L^d(A\cap B_\dist(R))$. As $A$ is a minimizer and $w_0\rho_0-w_1\rho_1>\delta> 0$ on $B_\dist(R)$, we can apply Proposition~\ref{generalInequality} to find that  
\begin{equation}D_\phi(A;B_\dist(R)) \le  D_\phi(B_\dist(R)^\comp;A) - \delta\mathcal L^d(A\cap B_\dist(R)) +w_0\rho_0(\widehat U_{11}) + w_1\rho_1(\widehat U_{1}) \label{eei2}\end{equation}
where \begin{align*}
   \widehat U_{1} &= \{x\in \tilde\Lambda^1_\phi(A)\cap \tilde\Lambda^0_\phi(B_\dist(R)): \phi(x;A\setminus B_\dist(R)) = 1\},\\
     \widehat U_{11} &= \{x\in \Lambda^0_{\phi}(A)\cap \Lambda^1_{\phi}(B_\dist(R)) : \phi(x;A\setminus B_\dist(R)) = 1\}.
\end{align*} 
By \eqref{Dphi1} we have $w_0\rho_0(\widehat U_{11})\le D_\phi(A;B_\dist(R))$. Combining the upper bound on $w_0\rho_0(\widehat U_{11})$ with \eqref{eei2} and simplifying, we find
\begin{equation}\mathcal L^d(A\cap B_\dist(R)) \le \frac{1}{\delta}D_\phi(B_\dist(R)^\comp;A) +\frac{ w_1}{\delta}\rho_1(\widehat U_{1}).\label{est without 11} \end{equation}

Recall that by definition, $\widehat U_1 \subset A\cap B_\dist(R)^\comp$. Additionally by Lemma~\ref{lem:phi isolation}, $\widehat U_1 \subset B_\dist(R+\e)$ as $\phi(x;A\setminus B_\dist(R)) = 1$ while $\phi(x;A) = 0$. Thus, $\widehat U_1 \subset A\cap (B_\dist(R+\e)\setminus B_\dist(R))$. In particular, \begin{equation}w_1\rho_1(\widehat U_1) \le w_1\rho_1(A\cap (B_\dist(R+\e)\setminus B_\dist(R))) \le \e\Per_\e(B_\dist(R)^\comp;A).\label{eqn:U1 estimate}\end{equation} Additionally by Lemma~\ref{deficit and eps per}, we have $D_\phi(B_\dist(R)^\comp;A) \le \e\Per_{\varepsilon}(B_\dist(R)^\comp;A)$. Applying \eqref{est for R} from Lemma~\ref{lem: eps per ball est}, \[\mathcal L^d(A\cap B_\dist(R))\le \frac{1}{\delta}D_\phi(B_\dist(R)^\comp;A) + \frac{w_1}{\delta}\rho_1(\widehat U_{1}) \le \frac{2}{\delta}\e\Per_{\varepsilon}(B_\dist(R)^\comp;A)\le \frac{2\alpha R^{d-1}}{\delta}\e\] for $\alpha$ independent of $R,\delta,\e$, and $x$ as in Lemma~\ref{lem: eps per ball est}. 

 Next, we want to find a radius $s_1\in (R/2,R)$ that will give an order $\e^2$ estimate for $\mathcal L^d(A\cap B_\dist(s_1))$. For $r \le R$, one has \[\mathcal L^d(A \cap B_\dist(r)) \le \frac{2}{\delta}\e \Per_\e(B_\dist(r)^\comp; A).\] We can argue by a discrete slicing argument like in Lemma~\ref{BR slices} to show that there exists an $s_1\in (R/2,R)$ such that \[\mathcal L^d(A\cap B_{\dist}(s_1))\le \frac{2}{\delta}\e\Per_\e(B_\dist(s_1)^\comp;A)\le 32\alpha M\frac{R^{d-2}}{\delta^2}\e^2.\] 

Iterating the argument as in Lemma~\ref{BR slices} yields an order $\varepsilon^{i+1}$ estimate of $\mathcal L^d(A\cap B_\dist(s_i))$ for $s_i\in (R/2^i,R/2^{i-1})$ and $2\le i \le \log_2(\frac{R}{4\e})$ (i.e. $\e \le \frac{R}{2^{i+2}}$). After $d$ iterations, we find
    \[\mathcal L^d(A\cap B_\dist(R/2^{d}))\le \left(\frac{C_{d+1}}{R\delta^{d+1}}\right)\varepsilon^{d+1}, \] where $C_{d+1}\ceq 2^{\frac{d(d+1)}{2}+3d+1}\alpha M^d$.

Finally, we must show that $\mathcal L^d(A\cap B(R/2^{d+2})) = 0$. Let $z\in A\cap B_\dist(\frac{R}{2^{d+1}}+\varepsilon)$. We must consider a region slightly outside of $B_\dist(R/2^{d+1})$ as the following argument needs to apply all points in the $\e$-dilation of $B_\dist(R/2^{d+1})$. We want to show that $z\in \Lambda_\phi^1(A)$. To do so, by Assumption 5 it will suffice to show that if $z \in A$ then $\mathcal L^d(A^\comp \cap B(z,\e)) > \beta \e^d$. 

Recall the estimate from the previous steps, $\mathcal L^d(A\cap B_\dist(R/2^d)) \le \left(\frac{C_{d+1}}{R\delta^{d+1}}\right)\e^{d+1}$. As long as $ \left(\frac{C_{d+1}}{R\delta^{d+1}}\right)\eps < (\omega_\dist - \beta)$, or in other words $\eps < (\omega_\dist - \beta) \left(\frac{R\delta^{d+1}}{C_{d+1}}\right) \ceq C$ then we have 
\[\mathcal L^d(A\cap B_\dist(R/2^d)) < (\omega_\dist -\beta) \e^d\] where $\beta$ is the constant from Assumption~\ref{phi assumptions}. Then as $B_\dist(z,\e) \subset B_\dist(R/2^d)$ we estimate
\[\mathcal L^d(A\cap B_\dist(z,\e)) + \mathcal L^d(A^\comp \cap B_\dist(z,\e))  = \omega_\dist\e^d \implies \mathcal L^d(A^\comp \cap B_\dist(z,\e)) > \beta \e^d.\] Hence, $z\in \Lambda_\phi^1(A)$. In particular, this means that $\tilde\Lambda_\phi^1(A)\cap B_\dist(\frac{R}{2^{d+1}}+\e) = \emptyset.$

We will now examine the difference in energies after removing $B_\dist(R/2^{d+1})$ in order to show that we must actually remove $B_\dist(R/2^{d+2})$ in order to achieve $\mathcal L^d(A \cap B_\dist(R/2^{d+2})) = 0$. By Corollary~\ref{cor:exact diff}, the difference in energy after removing the set $E=B_\dist(R/2^{d+1})$ from $A$ is 
\begin{align*}
J_{\phi}(A\setminus B_\dist(R/2^{d+1})) - J_{\phi}(A) &= w_1\rho_1(\widehat U_{1}\cup U_2\cup \widehat U_3) - (w_0\rho_0-w_1\rho_1)(\widetilde U_{3}\cup U_4) \\
& - w_0\rho_0(U_5 \cup \widetilde U_{6} \cup \widetilde U_{9}\cup \widetilde U_{10} \cup \widetilde U_{11}\cup U_{12}),
\end{align*}
where all sets are as defined in Table~\ref{table:lambda} and \eqref{eqn:U hat and tilde}. By construction, 
\[[\widehat U_{1}\cup U_2\cup \widehat U_3] \subset \left[\tilde\Lambda^1_\phi(A)\cap B_\dist\left(\frac{R}{2^{d+1}}+\varepsilon\right)\right]. \] However, we have just shown that  $\tilde\Lambda^1_\phi(A)\cap B_\dist(\frac{R}{2^{d+1}}+\varepsilon) = \emptyset$. Thus, we conclude that $\widehat U_{1} = U_2 = \widehat U_3 = \emptyset$.

As $w_0\rho_0 -w_1\rho_1> \delta > 0$ on $B_\dist(R/2^{d+1})$, the difference in energies becomes
\begin{align*}
J_{\phi}(A\setminus B_\dist(R/2^{d+1})) - J_{\phi}(A) &= - (w_0\rho_0-w_1\rho_1)(\widetilde U_{3}\cup U_4) 
 - w_0\rho_0(U_5 \cup \widetilde U_{6} \cup \widetilde U_{9}\cup \widetilde U_{10} \cup \widetilde U_{11}\cup U_{12})\\ 
 &\le -\delta\mathcal L^d(\widetilde U_{3}\cup U_4) - \delta\mathcal L^d(U_5 \cup \widetilde U_{6} \cup \widetilde U_{9}\cup \widetilde U_{10} \cup \widetilde U_{11}\cup U_{12})\\
 &\le 0.
\end{align*} By our assumption, $A$ is a minimizer, so $J_{\phi}(A\setminus B_\dist(R/2^{d+1})) - J_{\phi}(A) = 0$. This means all remaining sets must have measure zero, i.e. \[\mathcal L^d(\widetilde U_{3}) = \mathcal L^d(U_4) = \mathcal L^d(U_5)= \mathcal L^d(\widetilde U_{6}) = \mathcal L^d(\widetilde U_{9}) =\mathcal L^d(\widetilde U_{10}) = \mathcal L^d(\widetilde U_{11}) =\mathcal L^d(U_{12}) =  0.\] 
    
Recall from \eqref{eqn:intersection} that $A\cap B_\dist(R/2^{d+1}) = U_3 \cup U_4 \cup U_5 \cup U_6$. We have already shown that $U_3 = \widetilde U_{3}\cup \widehat U_3$, $U_4$, and $U_5$ all have measure zero. However, we notice that $\widehat U_6 \subset B(\frac{R}{2^{d+1}} + \eps) \setminus B(\frac{R}{2^{d+1}}-\eps)$, and so we can conclude that $\mathcal L^d(A \cap B(\frac{R}{2^{d+1}} - \eps)) = 0$. 

Then combining with the facts about $U_1, U_2,$ and $ U_3$, we then get that for any $s < \frac{R}{2^{d+1}} - \eps$ we have that $A\setminus B_\dist(s)$ is a minimizer of \eqref{DATP} and that $A \cap B_\dist(s)$ has measure zero.



\end{proof}

\begin{remark}
    As stated at the end of the proof, $A\setminus B_\dist(x,R/2^{d+2})$ is also a minimizer of \eqref{DATP}. In addition to providing a necessary condition for minimizers, Lemma~\ref{general slices} also gives a construction for a minimizer that is disjoint from $B_\dist(x,R/2^{d+2})$.
    
    Considering the assumptions, we \textit{cannot} relax the assumption that $A$ is a minimizer to $J_{\phi}(A\setminus B_\dist(x,r)) - J_{\phi}(A) \ge 0$ as we could for Lemma~\ref{BR slices} (see Remark~\ref{rem:relaxed assumptions}). Although the energy exchange inequality will still hold, we require that $A$ is a minimizer of \eqref{DATP} to show $\mathcal L^d(A\cap B_\dist(R/2^{d+2}))=0$. 
    \label{rem:gen min and assumption}
\end{remark} 

Assuming a minimizer to \eqref{DATP} exists, Lemma~\ref{general slices} allows us to show that minimizers are (a.e.) disjoint from certain sets where it is energetically advantageous to be assigned label 0 by the classifier. In Lemma~\ref{general l-inf conv}, we will use this result to show that for a prescribed distance $\eta >0$, there exists a minimizer of \eqref{DATP} that can be corralled to be within distance $\eta$ of any Bayes classifier for all $\varepsilon$ smaller than some threshold. This is the generalized version of Lemma~\ref{l-inf conv}. As we cannot expect minimizers of \eqref{DATP} to be sensitive to modification by a $\mathcal L^d$ measure zero set, we do not expect arbitrary minimizers to have this property. However from an arbitrary minimizer, Lemma~\ref{general l-inf conv} provides a method to construct a $\mathcal{L}^d$-a.e. equivalent minimizer that does satisfy this distance condition.

\begin{lemma}
    Let $A_0^{\max}$ be the maximal Bayes classifier, i.e. $A_0^{\max} = \{x\in \R^d: w_0\rho_0(x) \leq w_1\rho_1(x) \}$. Suppose $\rho_0,\rho_1>0$ and continuous and bounded from above on $\R^d$. Let $K$ be a compact set and fix $\eta > 0$. Then, there exists an $\e_0 > 0$ such that for any $0<\e\le \e_0$ and deterministic attack function $\phi$ satisfying Assumptions~\ref{phi consistency} and \ref{phi assumptions} for adversarial budget $\e$ such that for any minimizer $A_{\e,\phi}$ of the generalized adversarial training problem \eqref{DATP} there exist a $\mathcal L^d$ measure zero set $N^{\max}\in\cB(\R^d)$ such that \[\Big[(A_{\varepsilon,\phi}\setminus N^{\max})\cap K\Big]\subset\Big[(A_0^{\max})^\eta \cap K\Big].\] Furthermore, $A_{\varepsilon,\phi}\setminus N^{\max}$ is also a minimizer of \eqref{DATP}. 
    \label{general l-inf conv}
\end{lemma}

\begin{proof}
    We will follow the proof of Lemma~\ref{l-inf conv}. We again abuse notation and let $A_0 = A_0^{\max}$. Assume that $(A_0^\eta)^\comp \cap K \neq \emptyset$ as otherwise the result is trivial. The conditions are also trivially satisfied if $w_0\rho_0 - w_1\rho_1$ never changes sign.

    Let $R = \frac{\eta}{3}$. By the same argument as in Lemma~\ref{l-inf conv}, the continuity of $w_0\rho_0 - w_1\rho_1$ on the compact set $\overline{A_0^R}\cap \overline{K^{2R}}$ allows us to conclude that there exists a $\delta > 0$ such that $w_0\rho_0 - w_1\rho_1 > \delta$ on $\left(\overline{A_0^R}\right)^\comp \cap  \overline{K^{2R}}$ and $(A_0^\eta)^\comp \cap K$.
    
    

    As $(A_0^\eta)^\mathsf{c} \cap K$ is compact, there exists a finite covering of $(A_0^\eta)^\mathsf{c} \cap K$ by $\{B_\dist(x_i,R/2^{d+2})\}_{1\le i\le n}$ for some $n\in \N$ such that 
    \[\bigcup_{i=1}^n B_\dist(x_i,2R) \subset \Big[E_\delta^\comp \cap \overline{K^{2R}}\Big]\] where $E_{\delta} = \{x \in \R^d : w_0\rho_0(x) - w_1\rho_1(x) \le \delta\}$. Hence, each $B_\dist(x_i,2R)$ satisfies the conditions of Lemma~\ref{general slices} for $\delta$ from the continuity bound. As the constant $C$ from Lemma~\ref{general slices} is independent of $x$, we can let $\varepsilon_0 = \min\left\{R/2^{d+2}, CR\delta^{d+1}\right\}$. 
    
    Suppose $A_{\varepsilon,\phi}$ is a minimizer of the generalized adversarial training problem \eqref{DATP} for some $0<\varepsilon \le \varepsilon_0$ and let $N^{\max} = \bigcup_{i=1}^n  [A_{\varepsilon,\phi}\cap B_\dist(x_i, R/2^{d+2})]$. Then,
    \[ \mathcal L^d\left( \bigcup_{i=1}^n \left[A_{\varepsilon,\phi} \cap B_\dist(x_i, R/2^{d+2})\right]\right) \le \sum_{i=1}^n \mathcal L^d(A_{\varepsilon,\phi}\cap B_\dist(x_i, R/2^{d+2})) =0,\] so $N^{\max}$ is a $\mathcal L^d$ measure zero set.
    By Remark~\ref{rem:gen min and assumption}, an iterative application of Lemma~\ref{general slices} removing one norm ball at a time ensures that $A_{\varepsilon,\phi}\setminus N^{\max}$ is a minimizer of \eqref{DATP}. Furthermore, $[A_{\varepsilon,\phi}\setminus N^{\max}] \cap [(A_0^\eta)^\comp \cap K] = \emptyset$ by construction which implies
     \[\Big[(A_{\varepsilon,\phi}\setminus N^{\max})\cap K\Big]\subset\Big[A_0^\eta\cap K\Big].\]

\end{proof}

\begin{remark}
    In Lemma~\ref{general slices}, we require compactness both for the continuity argument and for the finite covering argument to ensure that we are removing a set of $\mathcal L^d$ measure zero. Compare this with Lemma~\ref{l-inf conv} and Remark~\ref{rem:compactness}.
\end{remark}

We can analogously show that (up to a set of $\mathcal L^d$ measure zero) we can corral $A_{\e,\phi}$ by an $\eta$-erosion of the minimal Bayes classifier $A_0^{\min}$ by considering the flipped density problem. We can apply the result from Lemma~\ref{general l-inf conv} to conclude that on a compact set $K\subset \R^d$ for a fixed $\eta > 0$, there exists a $\varepsilon_0 >0$ such that for $0<\varepsilon\le\varepsilon_0$ and a deterministic attack function $\phi$ with adversarial budget $\e$ satisfying the appropriate assumptions, then for any minimizer $A_{\varepsilon,\phi}$ of \eqref{DATP} there exist a $\mathcal L^d$ measure zero set $N^{\min}$ such that 
\[\Big[(A_0^{\min})^{-\eta}\cap K\Big] \subset \Big[(A_{\varepsilon,\phi}\cup N^{\min})  \cap K\Big].\] Observe that by construction $N^{\max}\subset A_0^\comp$ and $N^{\min} \subset A_0$ so the two sets are disjoint.
Like in the previous case, this establishes a two-sided, ``corralling'' bound on any minimizer for $\e$ small enough, namely,
\[\Big[(A_0^{\min})^{-\eta}\cap K\Big] \subset \Big[(A_{\varepsilon,\phi}\cup N^{\min}\setminus N^{\max})  \cap K\Big] \subset \Big[(A^{\max}_0)^\eta\cap K\Big].\]

\begin{remark}
    If the Bayes classifier is unique in the sense of Remark~\ref{rem:BC uniqueness}, then for any $\eta > 0$ and compact set $K\subset \R^d$, there exists an $\e_0 > 0$ such that for all $0<\varepsilon \le \varepsilon_0$ and $\phi$ satisfying Assumptions~\ref{phi consistency} and \ref{phi assumptions} for adversarial budget $\e$,
    \[\Big[(A_0)^{-\eta}\cap K\Big]\subset \Big[(A_{\varepsilon,\phi}\cup N^{\min}\setminus N^{\max}) \cap K\Big] \subset\Big[(A_0)^{\eta}\cap K\Big],\] provided that $A_{\e,\phi}$ exists.
\end{remark}

 Following the sequence of proofs in Section~\ref{sec:conv atp}, we will now use the corralling result from Lemma~\ref{general l-inf conv} to examine the distance between minimizers of the generalized adversarial training problem \eqref{DATP} and Bayes classifiers. The next theorem is the generalization of Theorem~\ref{hausdorff conv} and establishes uniform convergence in the Hausdorff distance. As previously stated, there is currently no proof of existence for minimizers of \eqref{DATP}, so this result should be seen as a type of a priori uniform convergence estimate. 

\begin{theorem}
   Let $\phi$ be a deterministic attack function satisfying Assumptions~\ref{phi consistency} and \ref{phi assumptions}. Suppose $\rho_0, \rho_1$ are continuous and bounded from above on $\R^d$. Additionally suppose $\{A_{\varepsilon_i,\phi}\}_{i\in \N}$ is a sequence of minimizers of the generalized adversarial training problem \eqref{DATP} with $\varepsilon_i \to 0^+$ as $i\to\infty$. For any compact set $K\subset \R^d$, there exist sequences $\{N^{\min}_i\}_{i\in\N}$ and $\{N^{\max}_i\}_{i\in\N}$ of $\mathcal L^d$ measure zero sets such that 
    \[\lim_{i\to\infty}\dist_H\big(((A_{\varepsilon_i,\phi}\setminus N^{\max}_i)\cup A_0^{\max}) \cap K, A_0^{\max}\cap K\big) = 0 \] 
    and \[\lim_{i\to\infty} \dist_H\big(((A_{\varepsilon_i,\phi}\cup N^{\min}_i)\cap A_0^{\min})\cap K,  A_0^{\min}\cap K\big)=0.\] 
    \label{general hausdorff conv}
\end{theorem}

\begin{proof}
 The proof is identical to that of Theorem~\ref{hausdorff conv} where $N^{\max}_i$ and $N_i^{\min}$ are as defined in the proof of Lemma~\ref{general l-inf conv}. 
 
 
\end{proof}

\begin{corollary}
     If the Bayes classifier $A_0$ is unique in the sense of Remark~\ref{rem:BC uniqueness}, then under the same assumptions as Theorem~\ref{hausdorff conv},
    \[ \lim_{i\to\infty} \dist_H((A_{\varepsilon_i,\phi} \cup N_i^{\min}\setminus N^{\max}_i) \cap K, A_0\cap K) = 0.\]
    \label{gen BC uniqueness conv}
\end{corollary}

\begin{proof}
    This follows directly from Theorem~\ref{general hausdorff conv}.
\end{proof}

Recall that Assumption~\ref{nondegen assumption} is a non-degeneracy assumption on the Bayes classifier $A_0$ that ensures that $\dist_H(A_0^{\max},A_0^{\min}) = 0$ and that $A_0$ is unique up to a set of $\mathcal L^d$ measure zero. If we assume that the Bayes classifier is non-degenerate, then it becomes natural to examine the rates of convergence.

\begin{corollary}
     Let $\phi$ be a deterministic attack function satisfying Assumptions~\ref{phi consistency} and \ref{phi assumptions}. Suppose Assumption~\ref{nondegen assumption} holds and that for every $\e >0$ there exists a minimizer $A_{\e,\phi}$ to the generalized adversarial training problem \eqref{DATP}. Additionally, suppose $\rho_0,\rho_1$ are continuous and bounded from above on $\R^d$. 
     For any compact set $K\subset \R^d$, there exist sequences $\{N^{\min}_i\}_{i\in\N}$ and $\{N^{\max}_i\}_{i\in\N}$ of $\mathcal L^d$ measure zero sets and a constant $C>0$ such that \[ \limsup_{i\to\infty}\frac{\dist_H ((A_{\varepsilon_i,\phi}\cup N_i^{\min}\setminus N^{\max}_i) \cap K, A_0\cap K)}{\varepsilon_i^{\frac{1}{d+2}}} \leq C \] where $A_0$ is the Bayes classifier.
     \label{cor:general roc}
\end{corollary}

\begin{proof}
    This proof is identical to that of Corollary~\ref{coro:eps conv rate}. 
    
    

    \end{proof}

\begin{remark}
    As in Lemma~\ref{coro:eps conv rate}, we expect that the convergence rate for minimizers of the generalized adversarial training problem \eqref{DATP} should be improved to $O(\e)$, but this would require more refined estimates than those available in Lemma~\ref{general slices}.

\end{remark}

\subsection{Application to the Probabilistic Adversarial Training Problem}\label{sec:patp app}

We now turn our attention the probabilistic adversarial training problem \eqref{PATP} which we will view as an instance of the generalized adversarial training problem \eqref{DATP}. In order to apply the results for \eqref{DATP} to \eqref{PATP}, we must verify that Assumptions~\ref{phi consistency} and \ref{phi assumptions} hold. In Remark~\ref{rem:verify assumption}, we established that \eqref{PATP} satisfies Assumption~\ref{phi consistency}; thus it only remains to show in the following proposition that \eqref{PATP} satisfies Assumption~\ref{phi assumptions}. 

\begin{proposition}
    Let $\e > 0$, $p \in [0,1)$, and $\{\mathfrak p_{x,\e}\}_{x\in\R^d}$ be a family of probability measures satisfying Assumption~\ref{p assumption}. The deterministic attack function $\phi_{\e,p}$ associated with the probabilistic adversarial training problem \eqref{PATP} satisfies Assumption~\ref{phi assumptions}.
    \label{verify assump prob}
\end{proposition}

\begin{proof}
   Suppose $x\in A^\comp$ such that $\mathcal L^d(A\cap B_\dist(x,\e)) > \beta\e^d$ for some $\beta >0$ to be determined. It will be sufficient to show that $\prob(x'\in A:x'\sim \mathfrak p_{x,\e}) > p$. Recall that we can express
   \begin{align*}
    \prob(x'\in A : x' \sim \mathfrak p_{x,\e}) &= \int_{\R^d} \e^{-d} \mathds{1}_{A}(x') \xi\left(\frac{x'-x}{\e}\right) \, dx' \\
    &= \int_{A \cap B_\dist(x,\e)} \e^{-d} \xi\left(\frac{x'-x}{\e}\right) \, dx'\\
   &> c\e^{-d}\mathcal L^d(A\cap B_\dist(x,\e))
\end{align*} where $c>0$ is the lower bound on $\xi$ from Assumption~\ref{p assumption}. If $\beta = \frac{p}{c}$, then $\prob(x'\in A : x' \sim \mathfrak p_{x,\e}) > p$ as desired. As the probabilistic adversarial training problem \eqref{PATP} satisfies the complement property, this is sufficient to conclude that Assumption~\ref{phi assumptions} holds for $\beta = \frac{p}{c}$.

\end{proof}

Since the probabilistic adversarial training problem \eqref{PATP} satisfies the requisite assumptions, we can state the following convergence result.

\begin{theorem}
    Suppose $\rho_0, \rho_1$ are continuous and bounded from above on $\R^d$ and fix $p \in [0,1)$. Additionally suppose $\{A_{\varepsilon_i,p}\}_{i\in \N}$ is a sequence of minimizers of the probabilistic adversarial training problem \eqref{PATP} with $\varepsilon_i \to 0^+$ as $i\to\infty$. For any compact set $K\subset \R^d$, there exist sequences $\{N^{\min}_i\}_{i\in\N}$ and $\{N^{\max}_i\}_{i\in\N}$ of measure zero sets such that 
    \[\lim_{i\to\infty}\dist_H\big(((A_{\varepsilon_i,p}\setminus N_i^{\max})\cup A_0^{\max}) \cap K, A_0^{\max}\cap K\big) = 0 \] 
    and \[\lim_{i\to\infty} \dist_H\big(((A_{\varepsilon_i,p}\cup N_i^{\min})\cap A_0^{\min})\cap K,  A_0^{\min}\cap K\big)=0.\] 
    \label{prob hausdorff conv}
\end{theorem}

When Assumption~\ref{nondegen assumption} holds, Theorem~\ref{prob hausdorff conv} asserts that 
\[\lim_{i\to\infty}(((A_{\e_i,p}\cup N_i^{\min}\setminus N^{\max}_i)\cup A_0)\cap K, A_0\cap K)= 0\] where $A_0$ is the unique Bayes classifier. Applying Corollary~\ref{cor:general roc} in this case, we find that the minimizers for the probabilistic training problem \eqref{PATP} converge to the Bayes classifier at the rate $O(\e^{\frac{1}{d+2}})$. 

We conclude the discussion of the probabilistic adversarial training problem \eqref{PATP} by commenting on why this results fails to extend to the $\Psi$-perimeter problem mentioned in Remark~\ref{rem: prev patp}.

\begin{remark}
    Recall that \textcite{bungert2023begins} consider the $\Psi$ adversarial training problem
    \begin{equation} \inf_{A\in \cB(\R^d} \E_{(x,y)\sim \mu} [|\mathds{1}_A(x) - y|] + \ProbPer_{\Psi}(A)
    \label{eqn:psi ATP}\end{equation} where the $\Psi$-perimeter is given by \[\ProbPer_{\Psi}(A)\ceq \int_{A^c} \Psi(\prob(x'\in A:x'\in \mathfrak p_{x})) d\rho_0(x) + \int_{A} \Psi(\prob(x'\in A^\comp:x'\in \mathfrak p_{x})) d\rho_1(x) \]  where $\Psi:[0,1]\to[0,1]$ is concave and non-decreasing. As opposed to the probabilistic adversarial training problem \eqref{PATP}, the existence of minimizers to \eqref{eqn:psi ATP} has been established.
    
    However, notice that the $\Psi$-perimeter \textit{cannot} be expressed as $w_0\rho_0(\Lambda_\Psi^0(A)) + w_1\rho_1(\Lambda_\Psi^1(A))$ if $\Psi$ is concave and non-decreasing as indicator functions are not concave. The $\Psi$-perimeter is an example of a data-adapted perimeter from the literature that cannot be represented via the deterministic attack framework. At present, whether the energy exchange inequality holds for the $\Psi$-perimeter remains an open question and proving this inequality for the $\Psi$-perimeter would be a promising first step towards showing uniform convergence of minimizers of \eqref{eqn:psi ATP}. 
    \label{rem:psi per}
\end{remark}

\section{Conclusion} \label{conclusion}

In this paper, we developed a unifying framework for the adversarial and probabilistic adversarial training problems to define more generalized adversarial attacks. Under natural set-algebraic assumptions, we derived the energy exchange inequality to quantify the effect of removing a set where a given label was energetically preferable from a minimizer.
Utilizing the energy exchange inequality to show that there exist minimizers disjoint from sets where the label 0 is strongly preferred energetically, we then proved uniform convergence in the Hausdorff distance for various adversarial attacks. This significantly strengthens the type of convergence established via $\Gamma$-convergence techniques \cite{bungert2024gammaconv}, as well as generalizing it to a broader class of adversarial attacks. Finally, we derived the rate of convergence based on our proof techniques. 

There are various future directions of research suggested by our results in this paper. First, the uniform convergence results increase the information that we have about minimizers and sequences of approximate minimizers. That information may be useful in establishing regularity results about minimizers, for example in the case of the adversarial training problem \eqref{eqn:ATP}, or may provide helpful information for proving existence in the generalized case. A different avenue of research to pursue would be to sharpen the convergence rates found in this paper by improving estimates from Lemmas~\ref{BR slices} and \ref{general slices} to determine whether the formally derived rate of $O(\e)$ can be achieved. Finally, one could consider how to expand the theoretical deterministic attack function framework to encapsulate other types of adversarial training problems such as $\Psi$ adversarial training problem \eqref{eqn:psi ATP}.


\section*{Acknowledgements}

The authors gratefully acknowledge the support of the NSF DMS 2307971 and the Simons Foundation TSM. The authors would also like to thank the reviewers for their insightful comments, in particular one reviewer who offered a simplified proof of Lemma~\ref{BR slices}.

\section{Appendix}
\subsection{The $U$ Sets for $\phi_\e$}
\label{app:phi eps claims}

In Remark~\ref{rem:lambda decomp}, we claim that further conclusions about the $U$ sets may be drawn when $\phi = \phi_\e$. We will now verify these claims. We consider only the cases where whether the entire set $U_i$ is attacked cannot be unambiguously determined by $\Lambda$-monotonicity (see Table~\ref{table:lambda}). In all of the following cases, we assume that the interaction of $A,E$ is nontrivial in the sense that $A\cap E$ and $A^\comp \cup E$ are both nonempty. Otherwise, the following sets will either be empty themselves or we trivially find $\dist(x,\emptyset)=\infty$. 

\begin{figure}
    \centering
    \scalebox{.75}{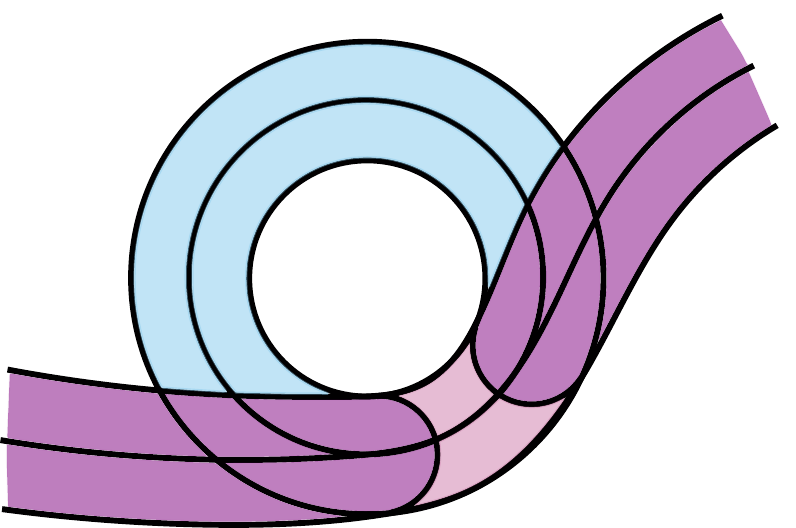}
    \caption{A degenerate example where $U_6$ and $U_9$ are neither solely attacked nor unattacked sets. The example arises because the boundaries of $A$ and $B_{\dist}(R)$ coincide. The pink and purple sets represent the $\e$-perimeter regions of $A$ whereas the blue and purple regions represent the $\e$-perimeter regions for $A\setminus \overline{B_\dist(R)}$.}
    \label{fig:U6 and U9}
\end{figure}

\begin{proposition}
    Let $\phi = \phi_\e$ and $A,E\in \mathcal B(\R^d)$. Then, $\widehat U_1 = \emptyset$.
\end{proposition} 

\begin{proof}
    
Suppose $x\in U_1\subset A\cap E^\comp$. By construction, we have $\dist(x,A^\comp) \ge \e$ and $\dist(x,E)\ge \e$. This implies that $\dist(x,A^\comp \cup E) = \dist(x,(A\setminus E)^\comp)\ge \e$ as well. Thus, $\widehat U_1 = \emptyset$. 

\end{proof}

\begin{proposition}
    Let $\phi = \phi_\e$ and $A,E\in \mathcal B(\R^d)$. Then, $\widetilde U_3 = \emptyset$.
\end{proposition} 

\begin{proof}
    
    Suppose $x\in U_3\subset A\cap E$. By construction, we have $\dist(x,A^\comp) \ge \e$ and $\dist(x,E^\comp)<\e$. As $\dist(x,A^\comp) \ge \e$, $B(x,\e)\subset A$. Furthermore, as $\dist(x,E^\comp) <\e$, $B(x,\e) \cap E^\comp \neq \emptyset$. Thus, there exists some $y \in B(x,\e)\cap E^\comp\subset A\cap E^\comp$. Hence, $\dist(x,A\setminus E) < \e$, so $\widetilde U_3 = \emptyset$.

    \end{proof}

\begin{proposition}
    Let $\phi = \phi_\e$ and $A,E\in \mathcal B(\R^d)$. Then, $\widetilde U_{10} = \emptyset$.
\end{proposition} 
\begin{proof}
    
  Suppose $x\in U_{10}$. By construction, we have $\dist(x,A)<\e$ and $\dist(x,E)\ge \e$. As $\dist(x,E)\ge \e$, $B(x,\e) \subset E^\comp$. Furthermore as $\dist(x,A) <\e$, $B(x,\e) \cap A \neq \emptyset$. Thus, there exists some $y \in B(x,\e)\cap A \subset A\cap E^\comp$. Hence, $\dist(x,A\setminus E) < \e$, so $\widetilde U_{10} = \emptyset$.

  \end{proof}

   As for $U_6, U_9$, and $U_{11}$, we can make no determinations about whether all points in these sets must be attacked or not. Figure~\ref{fig:ePerregions} shows an example where $U_{11}$ must be split into attacked and unattacked subsets. In special cases where the boundaries of the sets $A$ and $E$ coincide, $U_6$ and $U_9$ may also need to be split into attacked and unattacked subsets (see
   Figure~\ref{fig:U6 and U9}).

\subsection{$\Lambda$-Set Decompositions}
\label{app:Lambda sets}

For completeness, we give further details about the decompositions by $U$ sets in Proposition~\ref{generalInequality}, namely $\Lambda_\phi^0(A\setminus E), \Lambda_\phi^1(A\setminus E), A\cap E, A\setminus E, (A\setminus E)^\comp, D_\phi(A;E),$ and $D_\phi(E^\comp;A)$. Table~\ref{table:lambda} is reproduced for ease of reference.

\begin{table}[!ht]
\centering
\begin{tabular}{|c|c|}
\hline 
If $x\in U_i$ & Then for $\Lambda(A \backslash E)$ we have  \\
\hline 
$U_1\ceq\tilde\Lambda^1_\phi(A) \cap \tilde\Lambda^0_\phi(E)$ & N/A  \\
$U_2\ceq\tilde\Lambda^1_\phi(A) \cap {\Lambda^0_{\phi}}(E)$ & $x\in \Lambda^1_\phi(A\setminus E)$ \\ 
$U_3\ceq\tilde\Lambda^1_\phi(A) \cap {\Lambda^1_{\phi}}(E)$ & N/A  \\
$U_4\ceq\tilde\Lambda^1_\phi(A) \cap \tilde\Lambda^1_\phi(E)$ &  $x\in \tilde\Lambda_\phi^0(A\setminus E)$  \\ 
$U_5\ceq\Lambda^1_{\phi}(A) \cap \tilde\Lambda^1_\phi(E)$ &  $x\in \tilde\Lambda_\phi^0(A\setminus E)$ \\ 
$U_6\ceq \Lambda^1_{\phi}(A) \cap \Lambda^1_{\phi}(E)$ & N/A  \\
$U_7\ceq \Lambda^1_{\phi}(A) \cap \Lambda^0_{\phi}(E)$ & $x\in \Lambda_\phi^1(A\setminus E)$  \\ 
$U_8\ceq\Lambda^1_{\phi}(A) \cap \tilde\Lambda^0_\phi(E)$ & $x\in \Lambda_\phi^1(A\setminus E)$\\ 
$U_9\ceq{\Lambda^0_{\phi}}(A) \cap {\Lambda^0_{\phi}}(E)$ & N/A \\
$U_{10}\ceq{\Lambda^0_{\phi}}(A) \cap \tilde\Lambda^0_\phi(E)$ & N/A \\
$U_{11}\ceq\Lambda^0_{\phi}(A) \cap {\Lambda^1_{\phi}}(E)$ & N/A  \\
$U_{12}\ceq\Lambda^0_{\phi}(A) \cap \tilde\Lambda^1_\phi(E)$ & $x\in \tilde\Lambda_\phi^0(A\setminus E)$ \\
$U_{13}\ceq \tilde\Lambda^0_\phi(A)$ & $x\in \tilde\Lambda_\phi^0(A\setminus E)$ \\
\hline
\end{tabular}
\end{table}

\begin{itemize}
    \item $\Lambda_{\phi}^0(A\setminus E)$ is comprised of all $U$ sets such that $U_i\not\in A\setminus E$ and the points can be attacked by the adversary for the classifier $A\setminus E$. The $U_i\not\in A\setminus E$ are all $U$ sets such that the $\Lambda$-set for $A$ has the superscript 0 or the $\Lambda$-set for $E$ has the superscript 1, that is, $U_3, U_4, U_5, U_6, U_9, U_{10}, U_{11},U_{12}$ and $U_{13}$. However, $U_4, U_5, U_{12},$ and $U_{13}$ are all unattacked by Table~\ref{table:lambda}. Thus, $\Lambda_{\phi}^0(A\setminus E)$ contains the attacked subsets of $U_3, U_6, U_9, U_{10},$ and $U_{11}$. Hence,
    \[\Lambda_{\phi}^0(A\setminus E) = \widehat U_3\cup \widehat U_6\cup \widehat U_9\cup \widehat U_{10} \cup \widehat U_{11}.\]
    \item $\Lambda_{\phi}^1(A\setminus E)$ is comprised of all $U$ sets such that $U_i\in A\setminus E$ and the points can be attacked by the adversary for the classifier $A\setminus E$. The $U_i\not\in A\setminus E$ are all $U$ sets such that the $\Lambda$-set for $A$ has the superscript 1 and the $\Lambda$-set for $E$ has the superscript 0, that is, $U_1, U_2, U_7,$ and $U_8$. By Table~\ref{table:lambda}, the sets $U_2, U_7$, and $U_8$ are belong entirely to $\Lambda_{\phi}^1(A\setminus E)$, so $\Lambda_{\phi}^1(A\setminus E)$ contains those sets and the attacked subset of $U_1$. Hence,
    \[\Lambda_{\phi}^1(A\setminus E) = \widehat U_{1}\cup U_2\cup U_7\cup U_8.\]
    \item $A\cap E$ is comprised of all $U$ sets such that the $\Lambda$-sets for $A$ and $E$ both have the superscript $1$. Hence,
    \[A\cap E = U_3\cup U_4\cup U_5 \cup U_6.\]
    \item  $A\setminus E$ is comprised of all $U$ sets such that the $\Lambda$-set for $A$ has the superscript 1 and  the $\Lambda$-set for $E$ has the superscript $0$. Hence,
    \[A\setminus E = U_1\cup U_2\cup U_7\cup U_8.\]
    \item $(A\setminus E)^\comp$ is comprised of all $U$ sets not in $A\setminus E$, or alternatively, all $U$ sets such that either the $\Lambda$-set for $A$ has the superscript 0 or the $\Lambda$-set for $E$ has the superscript $1$. Hence,
    \[(A\setminus E)^\comp = U_3\cup U_4\cup U_5\cup U_6\cup U_9\cup U_{10}\cup U_{11}\cup U_{12}\cup U_{13}.\]
    \item Recall $D_\phi(A;E) = w_0\rho_0(\Lambda_\phi^0(A)\cap E) + w_1\rho_1(\Lambda_\phi^1(A)\cap E)$. The set $E$ can be expressed in terms of $\Lambda$-sets by $E =\Lambda_\phi^1(E) \cup \tilde \Lambda_\phi^1(E).$ By Table~\ref{table:lambda},
    \[ D_\phi(A;E)= w_0\rho_0(U_{11}\cup U_{12}) + w_1\rho_1(U_5\cup U_6).\]
    \item Recall $D_\phi(E^\comp;A) = w_0\rho_0(\Lambda_\phi^1(E)\cap A) + w_1\rho_1(\Lambda_\phi^0(E)\cap A)$ by the Complement Property of $\Lambda$-sets. The set $A$ can be expressed in terms of $\Lambda$-sets by $A = \Lambda_\phi^1(A)\cup \tilde\Lambda_\phi^1(A)$. By Table~\ref{table:lambda},
    \[ D_\phi(E^\comp;A)= w_0\rho_0(U_{3}\cup U_{6}) + w_1\rho_1(U_2\cup U_7).\]
\end{itemize}

\printbibliography

\end{document}